\documentclass[11pt,reqno]{amsart}
\usepackage{amsmath}
\usepackage{amssymb}
\usepackage{amsthm}
\usepackage{amsfonts}
\usepackage{appendix}
\usepackage{graphicx}
\usepackage{setspace}
\usepackage{color}

\newtheorem{theorem}{Theorem}[section]
\newtheorem{lemma}[theorem]{Lemma}
\newtheorem{proposition}[theorem]{Proposition}

\theoremstyle{definition}
\newtheorem{definition}[theorem]{Definition}
\newtheorem{remark}[theorem]{Remark}

\numberwithin{equation}{section}
\allowdisplaybreaks[4]
\date{}
\def \p {\partial}
\def \O {\varOmega}
\def \e {\lambda}
\def \r {T}
\def \ta {\theta}
\def \d {\cdot}
\def \g {\Phi}
\def \ga {\gamma}
\def \be {\beta}
\def \n {\nabla}
\def \nh {\nabla_h}
\def \la {\Delta}
\def \ma {\mathcal}%

\newcommand{\ie}{i.e.}
\newcommand{\eg}{e.g.}
\newcommand{\et}{\emph{et al.}}

\newcommand{\norm}[1]{\left\lVert#1\right\rVert}

\newcommand{\rt}[1]{\int^r_0\int_{\O}#1dxdydzdt}

\newcommand{\ty}[1]{\int^{\infty}_0\int_{\O}#1dxdydzdt}

\newcommand{\xkh}[1]{\left(#1\right)}
\newcommand{\zkh}[1]{\left[#1\right]}
\newcommand{\dkh}[1]{\left\{#1\right\}}

\newcommand{\ds}[1]{\int^t_0#1ds}
\newcommand{\dto}[1]{\int^{t^*_0}_0#1dt}

\newcommand{\dt}[1]{\int^r_0#1dt}
\newcommand{\dz}[1]{\int^1_{-1}#1dz}
\newcommand{\dk}[1]{\int^z_0#1d\xi}
\newcommand{\mm}[1]{\int_{M}#1dxdy}
\newcommand{\oo}[1]{\int_{\O}#1dxdydz}

\begin{document}
\title[Rigorous justification of the hydrostatic approximation]{The hydrostatic approximation of the Boussinesq equations with rotation in a thin domain}


\author[X. Pu]{Xueke Pu}
\address[X. Pu]{School of Mathematics and Information Science, Guangzhou University, Guangzhou 510006,~China}
\email{puxueke@gmail.com}

\author[W. Zhou]{Wenli Zhou}
\address[W. Zhou]{School of Mathematics and Information Science, Guangzhou University, Guangzhou 510006,~China}
\email{wywlzhou@163.com}

\begin{abstract}
In this paper, we improve the global existence result in \cite{cc2017} slightly. More precisely, the global existence of strong solutions to the primitive equations with only horizontal viscosity and diffusivity is obtained under the assumption of initial data $(v_0,\r_0) \in H^1$ with $\p_z v_0 \in L^4$. Moreover, we prove that the scaled Boussinesq equations with rotation strongly converge to the primitive equations with only horizontal viscosity and diffusivity, in the cases of $H^1$ initial data,~$H^1$ initial data with additional regularity $\p_z v_0 \in L^4$~and $H^2$ initial data, respectively, as the aspect ration parameter $\e$ goes to zero, and the rate of convergence is of the order $O(\e^{{\eta}/2})$ with $\eta=\min\{2,\be-2,\ga-2\}(2<\be,\ga<\infty)$. The convergence result implies a rigorous justification of the hydrostatic approximation.
\end{abstract}

\maketitle\vspace{-5mm}

\section{Introduction}
The primitive equations are considered as the fundamental model in geophysical flows (see, \eg, \cite{wm1986,jp1987,ds1996,am2003,gk2006}).~In large-scale ocean dynamics, an important feature is that the vertical scale of ocean is much smaller than the horizontal scale, which means that we have to use the hydrostatic approximation to simulate the motion of ocean in the vertical direction.~Owing to this fact and the high accuracy of hydrostatic approximation, the three-dimensional viscous primitive equations of ocean dynamics can be formally derived from the three-dimensional Boussinesq equations with rotation (see \cite{rt1992,ct2007}).

The small aspect ratio limit from the Navier-Stokes equations to the primitive equations was studied first by Az\'{e}rad-Guill\'{e}n\cite{pa2001} in a weak sense, then by Li-Titi\cite{lt2019} in a strong sense with error estimates, and finally by Furukawa \et\cite{kf2020} in a strong sense but under relaxing the regularity on the initial condition. ~Subsequently, the strong convergence of solutions of the scaled Navier-Stokes equations to the corresponding ones of the primitive equations with only horizontal viscosity was obtained by Li-Titi-Yuan\cite{yu2022}.~Furthermore, the rigorous justification of the hydrostatic approximation from the scaled Boussinesq equations to the primitive equations with full viscosity and diffusivity was obtained by Pu-Zhou\cite{pz2021}.

From a physical point of view, fluid flow is strongly influenced by effect of stratification (see, \eg, \cite{jp1987,am2003,gk2006}).~An important observation for effect of stratification is that the density of a fluid changes with depth.~In some mathematical studies, considering the hydrodynamic equations with density stratification term can often obtain better results (see, \eg, \cite{ct2012,cc2014,jl2014,es2016,cc2017,es2020}).~These two facts show that density stratification term is of great significance both physically and mathematically.~The rigorous mathematical derivation for the governed equations describing the motion of stable stratified fluid, \ie, the viscous primitive equations with density stratification, is due to the work of Pu-Zhou\cite{wz2021}.~Based on the ideas that follow from Li-Titi\cite{lt2019}, Li-Titi-Yuan\cite{yu2022}, and Pu-Zhou\cite{wz2021}, We study the hydrostatic approximation of the Boussinesq equations with rotation in a thin domain.

\subsection{The scaled Boussinesq equations with rotation in a thin domain}
Let $\O_\e=M\times(-\e,\e)$ be a $\e$-dependent domian, where $M=(0,1)\times(0,1)$.~Here, $\e=H/L$ is called the aspect ratio, measuring the ratio of the vertical and horizontal scales of the motion, which is usually very small.~Say, for large-scale ocean circulation, the ratio~$\e\sim 10^{-3}\ll 1$.

Denote by $\nh=(\p_x,\p_y)$ the horizontal gradient operator.~Then the horizontal Laplacian operator $\la_h$ is given by
\begin{equation*}
  \la_h=\nh \d \nh=\p_{xx}+\p_{yy}.
\end{equation*}Let us consider the anisotropic Boussinesq equations with rotation defined on $\O_\e$
\begin{equation}\label{eq:udn}
\begin{cases}
  \p_t u+(u \d \n)u+\n\pi-\ta\vec{k}+f_0\vec{k} \times v=\mu_h \la_h u+\mu_z \p_{zz}u,\\
  \p_t \ta+u \d \n\ta=\kappa_h\la_h\ta+\kappa_z\p_{zz}\ta,\\
  \n \d u=0,
\end{cases}
\end{equation}where the three dimensional velocity field $u=(v,w)=(v_1,v_2,w)$, the pressure $\pi$ and temperature $\ta$ are the unknowns.~$f_0$ is the Coriolis parameter and $\vec{k}=(0,0,1)$ is unit vector pointing to the $z$-direction.~$\mu_h$ and $\mu_z$ represent the horizontal and vertical viscosity coefficients, while $\kappa_h$~and~$\kappa_z$ represent the horizontal and vertical heat diffusion coefficients, respectively.

In fact,~the anisotropic Boussinesq equations with rotation (\ref{eq:udn}) have an elementary exact solution
\begin{equation*}
(u,\ta,\pi)=\xkh{0,\bar{\ta}(z),\bar{\pi}(z)}=\xkh{0,{N^2}z,\frac{N^2}{2}z^2},
\end{equation*}satisfying the hydrostatic approximation
\begin{equation*}
  \frac{d\bar{\pi}(z)}{dz}-\bar{\ta}(z)=0.
\end{equation*}Here, $N>0$ is called the buoyancy or Brunt-V\"ais\"al\"a frequency and denotes the strengthen of stable stratification, which implies that the density of a fluid decreases with height and lighter fluid is above heavier fluid. Assume that
\begin{gather*}
  p(x,y,z,t)=\pi(x,y,z,t)-\bar{\pi}(z),\\
  \r(x,y,z,t)=\ta(x,y,z,t)-\bar{\ta}(z).
\end{gather*}Then the anisotropic Boussinesq equations with rotation (\ref{eq:udn}) become
\begin{equation}\label{eq:muh}
\begin{cases}
  \p_t v+(v \d \nh)v+w \p_z v+\nh p+f_0\vec{k} \times v=\mu_h\la_h v+\mu_z\p_{zz}v,\\
  \p_t w+v \d \nh w+w\p_z w+\p_z p-\r=\mu_h\la_h w+\mu_z\p_{zz}w,\\
  \p_t \r+v \d \nh\r+w\p_z\r+N^2w
  =\kappa_h\la_h\r+\kappa_z\p_{zz}\r,\\
  \nh \d v+\p_z w=0.
\end{cases}
\end{equation}

Firstly, we transform the anisotropic Boussinesq equations with rotation (\ref{eq:muh}), defined on the $\e$-dependent domain $\O_\e$, to the scaled Boussinesq equations with rotation defined on a fixed domain.~To this end, we introduce the following new unknowns
\begin{gather*}
  u_{\e}=(v_\e,w_\e),~v_\e(x,y,z,t)=v(x,y,\e z,t),\\
  w_\e(x,y,z,t)=\frac{1}{\e}w(x,y,\e z,t),~p_\e(x,y,z,t)=p(x,y,\e z,t),\\
  \r_\e(x,y,z,t)=\e\r(x,y,\e z,t),~\bar{\pi}_\e(z)=\bar{\pi}(\e z),~\bar{\ta}_\e(z)=\e^2\bar{\ta}(\e z),
\end{gather*}for any $(x,y,z)\in\O=:M\times(-1,1)$ and for any $t\in(0,\infty)$.~Then the last two scalings allow us to write the pressure and temperature non-dimensionally as
\begin{equation*}
  \bar{\pi}_\e(z)+p_\e(x,y,z,t)=\bar{\pi}(\e z)+p(x,y,\e z,t)=\pi(x,y,\e z,t)
\end{equation*}and
\begin{equation*}
  \bar{\ta}_\e(z)+\e\r_\e(x,y,z,t)=\e^2(\bar{\ta}(\e z)+\r(x,y,\e z,t))=\e^2\ta(x,y,\e z,t),
\end{equation*}respectively.

Suppose that $\mu_h=\kappa_h=1$, $\mu_z=\e^{\be}$, and $\kappa_z=\e^{\ga}$, with $2 < \be,\ga<\infty$.~Under these scalings, the anisotropic Boussinesq equations with rotation (\ref{eq:muh}) defined on $\O_\e$ can be written as the following scaled Boussinesq equations with rotation
\begin{equation}\label{eq:bve}
\begin{cases}
  \p_t v_\e+(v_\e \d \nh)v_\e+w_\e \p_z v_\e+\nh p_\e+f_0\vec{k} \times v_\e=\la_h v_\e+\e^{\be-2}\p_{zz}v_\e,\\
  \e(\p_t w_\e+v_\e \d \nh w_\e+w_\e \p_z w_\e)+\frac{1}{\e}(\p_z p_\e-\r_\e)=\e\la_h w_\e+\e^{\be-1}\p_{zz}w_\e,\\
  \frac{1}{\e}\xkh{\p_t \r_\e+v_\e \d \nh \r_\e+w_\e \p_z \r_\e}+\e N^2 w_\e=\frac{1}{\e}\la_h\r_\e+\e^{\ga-3}\p_{zz}\r_\e,\\
  \nh \d v_\e+\p_z w_\e=0,
\end{cases}
\end{equation}defined on the fixed domain $\O$.

Set $\e^2 N^2=1$, \ie, $N \sim 1/\e$, which means that the stratification effect is very strong as the aspect ratio $\e$ tends to zero.~In such a case, the scaled Boussinesq equations with rotation (\ref{eq:bve}) can be rewritten as
\begin{equation}\label{eq:ptv}
\begin{cases}
  \p_t v_\e-\la_h v_\e-\e^{\be-2}\p_{zz}v_\e+(v_\e \d \nh)v_\e+w_\e \p_z v_\e+\nh p_\e+f_0\vec{k} \times v_\e=0,\\
  \e^2\xkh{\p_tw_\e-\la_h w_\e-\e^{\be-2}\p_{zz}w_\e+v_\e \d \nh w_\e+w_\e \p_z w_\e}+\p_zp_\e-\r_\e=0,\\
  \p_t \r_\e-\la_h \r_\e-\e^{\ga-2}\p_{zz}\r_\e+v_\e \d \nh \r_\e+w_\e \p_z \r_\e+w_\e=0,\\
  \nh \d v_\e+\p_z w_\e=0.
\end{cases}
\end{equation}

Next, we supply the scaled Boussinesq equations with rotation (\ref{eq:ptv}) with the following boundary and initial conditions
\begin{gather}
  v_\e,w_\e,p_\e~\textnormal{and}~\r_\e~\textnormal{are periodic in}~x,y,z, \label{ga:are}\\
  (v_\e,w_\e,\r_\e)|_{t=0}=(v_0,w_0,\r_0), \label{ga:vew}
\end{gather}where $(v_0,w_0,\r_0)$ is given.~Moreover, we also equip the system (\ref{eq:ptv}) with the following symmetry condition
\begin{equation}\label{eq:eve}
  v_\e,w_\e,p_\e~\textnormal{and}~\r_\e~\textnormal{are even,~odd,~even and odd with respect to}~z,~\textnormal{respectively}.
\end{equation}Noting that the above symmetry condition is preserved by the scaled Boussinesq equations with rotation (\ref{eq:ptv}), \ie, it holds provided that the initial data satisfies this symmetry condition.~Due to this fact, throughout this paper, we always suppose that
\begin{equation}\label{eq:xyz}
  v_0,w_0~\textnormal{and}~\r_0~\textnormal{are periodic in}~x,y,z,~\textnormal{and are even,~odd~and odd in}~z,~\textnormal{respectively}.
\end{equation}

In this paper, we will not distinguish in notation between spaces of scalar and vector-valued functions.~Namely, we will use the same notation to denote both a space itself and its finite product spaces.~For simplicity, we denote by notation $\norm{\d}_p$ the $L^p(\O)$ norm.

The weak solutions of the scaled Boussinesq equations with rotation (\ref{eq:ptv}) are defined as follows.
\begin{definition}
Given $(u_0,\r_0)=(v_0,w_0,\r_0) \in L^2(\O)$, with $\n \d u_0=0$.~We say that a space periodic function $(v_\e,w_\e,\r_\e)$ is a weak solution of the system (\ref{eq:ptv}), subject to boundary and initial conditions (\ref{ga:are})-(\ref{ga:vew}) and symmetry condition~(\ref{eq:eve}), if\\
(i) $(v_\e,w_\e,\r_\e) \in C([0,\infty);L^2(\O)) \cap L^2_{loc}([0,\infty);H^1(\O))$ and\\
(ii) $(v_\e,w_\e,\r_\e)$ satisfies the following integral equality
\begin{flalign*}
  &\ty{\bigg\{\zkh{-v_\e \d \p_t \varphi_h-\e^2 w_\e \p_t\varphi_3-\r_\e \p_t \psi-\r_\e \varphi_3+w_\e \psi+f_0(\vec{k} \times v_\e) \d \varphi_h}\\
  &\qquad\quad+\zkh{\nh v_\e:\nh \varphi_h+\e^{\be-2}(\p_z v_\e) \d \p_z\varphi_h+\e^2\nh w_\e \d \nh \varphi_3}\\
  &\qquad\quad+\zkh{\e^\be(\p_z w_\e)\p_z\varphi_3+\nh \r_\e \d \nh \psi+\e^{\ga-2}(\p_z\r_\e)\p_z\psi}\\
  &\qquad\quad+\zkh{(u_\e \d \n)v_\e \d \varphi_h+\e^2(u_\e \d \n w_\e)\varphi_3+(u_\e \d \n \r_\e)\psi}\bigg\}}\\
  &\qquad=\oo{\xkh{v_0 \d \varphi_h(0)+\e^2w_0\varphi_3(0)+\r_0\psi(0)}},
\end{flalign*}for any spatially periodic function $(\varphi,\psi)=(\varphi_h,\varphi_3,\psi)$, with~$\varphi_h=(\varphi_1,\varphi_2)$, such that $\n \d \varphi=0$ and $(\varphi,\psi) \in C^{\infty}_c(\overline{\O}\times[0,\infty))$.
\end{definition}

\begin{remark}\label{re:ws}
The proof of the existence of weak solutions to the scaled Boussinesq equations with rotation (\ref{eq:ptv}) follows from the
similar argument in Lions-Temam-Wang \cite[Part IV]{rt1992}.~Specifically, for any initial data
$(u_0,\r_0)=(v_0,w_0,\r_0) \in L^2(\O)$, with $\n \d u_0=0$, we can prove that there exists a global weak solution
$(v_\e,w_\e,\r_\e)$ of the scaled Boussinesq equations with rotation (\ref{eq:ptv}), subject to boundary and initial conditions
(\ref{ga:are})-(\ref{ga:vew}) and symmetry condition (\ref{eq:eve}).~Moreover, by the similar argument as Lions-Temam-Wang
\cite[Part IV]{rt1992}, we can also show that it has a unique local strong solution $(v_\e,w_\e,\r_\e)$
for initial data $(u_0,\r_0)=(v_0,w_0,\r_0) \in H^1(\O)$, with $\n \d u_0=0$.
\end{remark}

\begin{remark}\label{re:lh}
Similar to the theory of three-dimensional Navier-Stokes equations, \eg, see Temam\cite[Ch.III, Remark 4.1]{rt1977} and Robinson \et\cite[Theorem 4.6]{jc2016}, we can prove that the weak solution $(v_\e,w_\e,\r_\e)$ satisfies the following energy inequality
\begin{flalign}
  &\frac{1}{2}\xkh{\norm{v_\e(t)}^2_2+\e^2\norm{w_\e(t)}^2_2+\norm{\r_\e(t)}^2_2}\nonumber\\
  &\qquad\quad+\ds{\xkh{\norm{\nh v_\e}^2_2+\e^{\be-2}\norm{\p_z v_\e}^2_2+\e^2\norm{\nh w_\e}^2_2}}\nonumber\\
  &\qquad\quad+\ds{\xkh{\e^\be\norm{\p_z w_\e}^2_2+\norm{\nh \r_\e}^2_2+\e^{\ga-2}\norm{\p_z \r_\e}^2_2}}\nonumber\\
  &\qquad\leq\frac{1}{2}\xkh{\norm{v_0}^2_2+\e^2\norm{w_0}^2_2+\norm{\r_0}^2_2}.\label{fl:vwt}
\end{flalign}for a.e.~$t \in [0,\infty)$, as long as the weak solution $(v_\e,w_\e,\r_\e)$ is obtained by Galerkin method.
\end{remark}

\subsection{The limiting system of the scaled Boussinesq equations with rotation}
Now we discuss the limiting system of the scaled Boussinesq equations with rotation (\ref{eq:ptv}).

When $2<\be,\ga<\infty$, taking the limit $\e \rightarrow 0$ in system (\ref{eq:ptv}), then this system formally converges to the following primitive equations with only horizontal viscosity and diffusivity
\begin{equation}\label{eq:ptl}
\begin{cases}
  \p_t v-\la_h v+(v \d \nh)v+w \p_z v+\nh p+f_0\vec{k} \times v=0,\\
  \p_z p-\r=0,\\
  \p_t \r-\la_h \r+v \d \nh \r+w \p_z \r+w=0,\\
  \nh \d v+\p_z w=0.
\end{cases}
\end{equation}

The term $w$ in the third equation of system (\ref{eq:ptl}) is called the density stratification term, providing additional dissipation for this system.~Moreover, we supply the limiting system (\ref{eq:ptl}) with the same boundary and initial conditions (\ref{ga:are})-(\ref{ga:vew}) and symmetry condition (\ref{eq:eve}) as the system (\ref{eq:ptv}). In studying the well-posedness of system (\ref{eq:ptl}), we observe that it is not necessary to give the initial condition for vertical velocity $w$, since there is no dynamic equation for vertical velocity in the system. So we say that the system (\ref{eq:ptl}) satisfies the initial condition (\ref{ga:vew}) just for convenience. Note that the initial value $w_0$ for vertical velocity $w_\e$ is uniquely determined by the divergence-free condition and symmetry condition (\ref{eq:eve}). Hence it can be represented as
\begin{equation}\label{eq:yxi}
  w_0(x,y,z)=-\dk{\nh \d v_0(x,y,\xi)},
\end{equation}for any~$(x,y) \in M$~and~$z \in (-1,1)$. Therefore, only the initial conditions of horizontal velocity and temperature are given throughout the paper.

We remark that the limiting system of the scaled Boussinesq equations with rotation (\ref{eq:ptv}) is the primitive equations with full viscosity and diffusivity when $\be=\ga=2$. This case was studied by the authors (see \cite{wz2021}). In consequence, the aim of this paper is to prove that the scaled Boussinesq equations with rotation (\ref{eq:ptv}) strongly converge to the primitive equations with only horizontal viscosity and diffusivity (\ref{eq:ptl}), in the cases of $H^1$ initial data,~$H^1$ initial data with additional regularity $\p_z v_0 \in L^4$ and $H^2$ initial data, respectively, as the aspect ration parameter tends to zero. These convergence results are briefly described as follows.
\begin{itemize}
  \item For $H^1$ initial data, the system (\ref{eq:ptl}) corresponding to (\ref{ga:are})-(\ref{eq:eve}) has a unique local strong solution $(v,\r)$ (see \cite{cc2017}).~Based on this local well-posedness result and Remark \ref{re:ws}, we obtain the local strong convergence theorem (see Theorem \ref{th:VTC}).
  \item For $H^1$ initial data with additional regularity $\p_z v_0 \in L^4$, the global existence of strong solutions to the system (\ref{eq:ptl}) with (\ref{ga:are})-(\ref{eq:eve}) is proved (see Theorem \ref{th:J2t}). Compared with \cite{cc2017}, the condition $(v_0,\r_0) \in L^\infty$ is removed by establishing $L^\infty_t L^4_x$ estimate on the vertical derivative of horizontal velocity that does not depend on $\norm{(v_0,\r_0)}_\infty$.~Consequently, the improved result and Remark \ref{re:ws} yield the global strong convergence theorem (see Theorem \ref{th:J3t}).
  \item For $H^2$ initial data, there exists a unique global strong solution $(v,\r)$ to the system (\ref{eq:ptl}) subject to (\ref{ga:are})-(\ref{eq:eve}) (see \cite{es2016}).~According to the energy estimate on the global strong solutions and Remark \ref{re:ws}, we establish the corresponding global strong convergence result (see Theorem \ref{th:J7t}).
\end{itemize}

As can be seen from above, the well-posedness results of primitive equations with only horizontal viscosity and diffusivity (\ref{eq:ptv}) will play an important role in proving that the system (\ref{eq:ptv}) strongly converges to the system (\ref{eq:ptl}) as the aspect ration parameter tends to zero.~In order to make full use of these known well-posedness results, we must construct the primitive equations with density stratification.~The way to construct it is to look for a suitable exact solution of the system (\ref{eq:udn}), \ie,
\begin{equation*}
(u,\ta,\pi)=\xkh{0,\bar{\ta}(z),\bar{\pi}(z)}=\xkh{0,{N^2}z,\frac{N^2}{2}z^2},
\end{equation*}which satisfies the hydrostatic approximation
\begin{equation*}
  \frac{d\bar{\pi}(z)}{dz}-\bar{\ta}(z)=0.
\end{equation*}In addition to these well-posedness results,~more results on the case of partial dissipation can be found in the work of Cao-Titi\cite{ct2012}, Fang-Han\cite{dy2020}, Li-Yuan\cite{li2022}, and Cao-Li-Titi\cite{cc2014,jl2014,es2020}.

Some other results for the primitive equations are as follows.~The global existence of weak solutions of the primitive equations with full viscosity and diffusivity was first given by Lions-Temam-Wang\cite{rt1992,jl1992,sw1995},~but the question of uniqueness to this mathematical model is still unclear.~Only in some special cases are known results (see \cite{db2003,tt2010,ik2014,jl2017,ju2017}).~For arbitrarily large initial data belonging to $H^1$, the global existence of strong solutions of the full primitive equations with Neumann boundary conditions was obtained by Cao-Titi\cite{ct2007}.~In the case of mixed Dirichlet and Neumann boundary conditions, this result was also proved by Kukavica-Ziane\cite{ik2007,mz2007}.~Considering the same boundary conditions, the existence of global strong solutions of the primitive equations without temperature was established by Hieber-Kashiwabara\cite{mh2016} for $L^p$ initial data, and later by Giga \et \cite{ym2020} for initial data in anisotropic $L^p$ space.~In addition,~the well-posedness result corresponding to the primitive equations without temperature in $L^p$ space can be extended to the full primitive equations, see Hieber \et \cite{ah2016}.

The inviscid primitive equations without temperature is called the hydrostatic Euler equations.~Kukavica \et \cite{ku2016} show that the solutions of the primitive equations converge to the solutions of the hydrostatic Euler equations,~as viscosity coefficient goes to zero.~The inviscid primitive equations with or without rotation is known to be ill-posed in Sobolev spaces, and its smooth solutions may develop singularity in finite time, see Renardy\cite{mr2009}, Han-Kwan and Nguyen\cite{dh2016}, Ibrahim-Lin-Titi\cite{si2021}, Wong\cite{tk2015}, and Cao \et \cite{cc2015}.~However, the local well-posedness of the inviscid primitive equations was established by Kukavica \et \cite{ku2011} for initial data in the space of analytic function, in which the maximal existence time of the analytical solutions depends on the rate of rotation $|f_0|$.~Subsequently, this local well-posedness result was improved by Ghoul \et \cite{te2022}, and then the long time existence of solutions was obtained.~For more results on the inviscid primitive equations, we refer to the work of Brenier\cite{yb1999}, Masmoudi-Wong\cite{nt2012}, and Kukavica \et \cite{nm2014}.

The rest of this paper is organized as follows.~Our main results are stated in Section 2.~The strong convergence results of $H^1$ initial data,~$H^1$ initial data with additional regularity,~and $H^2$ initial data are presented in Section 3,~section 4,~and Section 5,~respectively.

\section{Main results}
Now we are to state the main results of this paper.~Assume that initial data $(v_0,\r_0) \in H^1(\O)$.~Then it deduces from (\ref{eq:yxi}) that $(v_0,w_0,\r_0) \in L^2(\O)$, which implies that the system (\ref{eq:ptv}) subject to (\ref{ga:are})-(\ref{eq:eve}) has a global weak solution $(v_\e,w_\e,\r_\e)$ by Remark \ref{re:ws}, and the system (\ref{eq:ptl}) exists a unique local strong solution $(v,\r)$ (see \cite{cc2017}). Denote by $t^*_0$ the maximal existence time of the local strong solution $(v,\r)$ to the system (\ref{eq:ptl}). For this case, we have the following strong convergence theorem.

\begin{theorem}\label{th:VTC}
Given a periodic function pair $(v_0,T_0) \in H^1(\O)$ with $\dz{\nh \d v_0}=0$. Suppose that $(v_\e,w_\e,\r_\e)$ is a global weak solution of the system (\ref{eq:ptv}), satisfying the energy inequality (\ref{fl:vwt}), and that $(v,\r)$ is the unique local strong solution of the system (\ref{eq:ptl}), with the same boundary and initial conditions (\ref{ga:are})-(\ref{ga:vew}) and symmetry condition (\ref{eq:eve}). Let
\begin{equation*}
  (V_\e,W_\e,\g_\e)=(v_\e-v,w_\e-w,\r_\e-\r).
\end{equation*}Then the following estimate holds
\begin{flalign*}
  &\sup_{0 \leq t < t^*_0}\xkh{\norm{(V_\e,\e W_\e,\g_\e)}^2_2}(t)+\dto{\xkh{\norm{\nh V_\e}^2_2+\e^{\be-2}\norm{\p_z V_\e}^2_2}}\\
  &\qquad\quad+\dto{\xkh{\e^2\norm{\nh W_\e}^2_2+\norm{\nh \g_\e}^2_2+\e^\be\norm{\p_z W_\e}^2_2+\e^{\ga-2}\norm{\p_z \g_\e}^2_2}}\\
  &\qquad\leq C\e^\eta(t^*_0+1)e^{C(t^*_0+1)}\zkh{1+\xkh{\norm{v_0}^2_2+\norm{w_0}^2_2+\norm{\r_0}^2_2}^2},
\end{flalign*}where $\eta=\min\{2,\be-2,\ga-2\}$ with $2<\be,\ga<\infty$, and $C$ is a positive constant that does not depend on $\e$. As a result, we have the following strong convergences
\begin{gather*}
  (v_\e,\e w_\e,\r_\e) \rightarrow (v,0,\r),~in~L^{\infty}\xkh{[0,t^*_0);L^2(\O)},\\
  \xkh{\nh v_\e,\e^{(\be-2)/2}\p_z v_\e,\e \nh w_\e,w_\e} \rightarrow (\nh v,0,0,w),~in~L^2\xkh{[0,t^*_0);L^2(\O)},\\
  \xkh{\nh \r_\e,\e^{\be/2}\p_z w_\e,\e^{(\ga-2)/2}\p_z \r_\e} \rightarrow (\nh \r,0,0),~in~L^2\xkh{[0,t^*_0);L^2(\O)},
\end{gather*}and the rate of convergence is of the order $O(\e^{\eta/2})$.
\end{theorem}

The authors in \cite{cc2017} obtain the global strong solutions of system (\ref{eq:ptl}), provided that initial data $(v_0,T_0) \in H^1(\O)$ has the additional regularity that $\p_z v_0 \in L^q(\O)$ with $q \in (2,\infty)$ and $(v_0,\r_0) \in L^\infty(\O)$.~We improve this result slightly and then give the following theorem.
\begin{theorem}\label{th:J2t}
Assume that $(v_0,\r_0) \in H^1(\O)$ with $\dz{\nh \d v_0}=0$, and that $\p_z v_0 \in L^4(\O)$. Then there exists a unique global strong solution $(v,\r)$ to the system (\ref{eq:ptl}) subject to boundary and initial conditions (\ref{ga:are})-(\ref{ga:vew}) and symmetry condition (\ref{eq:eve}) such that the following energy estimate holds
\begin{flalign*}
  &\sup_{0 \leq s \leq t}\xkh{\norm{(v,\r)}^2_{H^1(\O)}}(s)+\ds{\norm{\nh v}^2_{H^1(\O)}}\\
  &\qquad+\ds{\xkh{{\norm{\nh \r}^2_{H^1(\O)}}+\norm{(\p_t v,\p_t \r)}^2_{L^2(\O)}}} \leq \ma{J}_2(t),
\end{flalign*}for any $t \in [0,\infty)$, where $\ma{J}_2(t)$ is a nonnegative continuously increasing function defined on $[0,\infty)$.
\end{theorem}

\begin{remark}
The result of Theorem \ref{th:J2t} still holds for $\p_z v_0 \in L^m(\O)$ with $m \in (4,\infty]$, since $\O=(0,1)^2\times(-1,1)$ is a set of finite measure.~For the case of $(v_0,\r_0) \in H^1(\O)$ with $\p_z v_0 \in L^m(\O)(2 < m < 4)$, the global existence of strong solutions to the system (\ref{eq:ptl}) is unknown without assuming that $(v_0,\r_0) \in L^\infty(\O)$.
\end{remark}

Based on the global existence of strong solutions in Theorem \ref{th:J2t}, we have the corresponding global strong convergence theorem.
\begin{theorem}\label{th:J3t}
Given a periodic function pair $(v_0,T_0) \in H^1(\O)$ satisfying $\dz{\nh \d v_0}=0$ and $\p_z v_0 \in L^4(\O)$. Suppose that $(v_\e,w_\e,\r_\e)$ is a global weak solution of the system (\ref{eq:ptv}), satisfying the energy inequality (\ref{fl:vwt}), and that $(v,\r)$ is the unique global strong solution of the system (\ref{eq:ptl}), with the same boundary and initial conditions (\ref{ga:are})-(\ref{ga:vew}) and symmetry condition (\ref{eq:eve}). Let
\begin{equation*}
  (V_\e,W_\e,\g_\e)=(v_\e-v,w_\e-w,\r_\e-\r).
\end{equation*}Then the following estimate holds
\begin{flalign*}
  &\sup_{0 \leq t \leq \ma{T}}\xkh{\norm{(V_\e,\e W_\e,\g_\e)}^2_2}(t)+\int^\ma{T}_0{\xkh{\norm{\nh V_\e}^2_2+\e^{\be-2}\norm{\p_z V_\e}^2_2}}dt\\
  &\qquad\quad+\int^\ma{T}_0{\xkh{\e^2\norm{\nh W_\e}^2_2+\norm{\nh \g_\e}^2_2+\e^\be\norm{\p_z W_\e}^2_2+\e^{\ga-2}\norm{\p_z \g_\e}^2_2}}dt\leq \e^\eta\ma{J}_3(\ma{T}),
\end{flalign*}for any $\ma{T}>0$, where $\eta=\min\{2,\be-2,\ga-2\}$ with $2<\be,\ga<\infty$,~and $\ma{J}_3(t)$ is a nonnegative continuously increasing function defined on $[0,\infty)$.~Therefore, the local strong convergences in Theorem \ref{th:VTC} can be extended to the global strong convergences.
\end{theorem}

Finally, we suppose that initial data $(v_0,\r_0)$ belongs to $H^2(\O)$.~Then from (\ref{eq:yxi}) it follows that $(v_0,w_0,\r_0)$ belongs to $H^1(\O)$. ~According to Remark \ref{re:ws}, there exists a unique local strong solution $(v_\e,w_\e,\r_\e)$ to the system (\ref{eq:ptv}) corresponding to (\ref{ga:are})-(\ref{eq:eve}). Moreover, the system (\ref{eq:ptl}) has a unique global strong solution $(v,\r)$ (see \cite{es2016}).~In this case, we also have the following strong convergence theorem.
\begin{theorem}\label{th:J7t}
Given a periodic function pair $(v_0,\r_0) \in H^2(\O)$ satisfying $\dz{\nh \d v_0}=0$. Suppose that $(v_\e,w_\e,\r_\e)$ is the unique local strong solution of the system (\ref{eq:ptv}), and that $(v,\r)$ is the unique global strong solution of the system (\ref{eq:ptl}), with the same boundary and initial conditions (\ref{ga:are})-(\ref{ga:vew}) and symmetry condition (\ref{eq:eve}). Let
\begin{equation*}
  (V_\e,W_\e,\g_\e)=(v_\e-v,w_\e-w,\r_\e-\r).
\end{equation*}Then, for any finite time $\ma{T}>0$, there is a small positive constant $\e(\ma{T})=\xkh{\frac{3\ell^2_0}{8\ma{J}_6(\ma{T})}}^{1/\eta}$ such that the system (\ref{eq:ptv}) exists a unique strong solution $(v_\e,w_\e,\r_\e)$ on the time interval $[0,\ma{T}]$, and that the system (\ref{fl:Ve0})-(\ref{fl:zWe}) (see Section 5, below) has the following estimate
\begin{flalign*}
  &\sup_{0 \leq t \leq \ma{T}}\xkh{\norm{(V_\e,\e W_\e,\g_\e)}^2_{H^1}}(t)+\int^\ma{T}_0{\xkh{\norm{\nh V_\e}^2_{H^1}+\e^{\be-2}\norm{\p_z V_\e}^2_{H^1}}}dt\\
  &\quad+\int^\ma{T}_0{\xkh{\e^2\norm{\nh W_\e}^2_{H^1}+\norm{\nh \g_\e}^2_{H^1}+\e^\be\norm{\p_z W_\e}^2_{H^1}+\e^{\ga-2}\norm{\p_z \g_\e}^2_{H^1}}}dt
  \leq \e^\eta\ma{J}_7(\ma{T}),
\end{flalign*}provided that $\e \in (0,\e(\ma{T}))$, where $\eta=\min\{2,\be-2,\ga-2\}$ with $2<\be,\ga<\infty$,~and $\mathcal{J}_7(t)$ is a nonnegative continuously increasing function that does not depend on $\e$.~As a result, we have the following strong convergences
\begin{gather*}
  (v_\e,\e w_\e,\r_\e) \rightarrow (v,0,\r),~in~L^{\infty}\xkh{[0,\ma{T}];H^1(\O)},\\
  \xkh{\nh v_\e,\e^{(\be-2)/2}\p_z v_\e,\e \nh w_\e,w_\e} \rightarrow (\nh v,0,0,w),~in~L^2\xkh{[0,\ma{T}];H^1(\O)},\\
  \xkh{\nh \r_\e,\e^{\be/2}\p_z w_\e,\e^{(\ga-2)/2}\p_z \r_\e} \rightarrow (\nh \r,0,0),~in~L^2\xkh{[0,\ma{T}];H^1(\O)},\\
  w_\e \rightarrow w,~in~L^{\infty}\xkh{[0,\ma{T}];L^2(\O)},
\end{gather*}and the rate of convergence is of the order $O(\e^{\eta/2})$.
\end{theorem}

\section{Strong convergence for $H^1$ initial data}
In this section, assume that initial data $(v_0,\r_0) \in H^1(\O)$, where initial velocity $v_0$ satisfies
\begin{equation*}
  \dz{\nh \d v_0(x,y,z)}=0,~\textnormal{for all}~(x,y)\in M,
\end{equation*}we prove that the scaled Boussinesq equations with rotation (\ref{eq:ptv}) strongly converge to the primitive equations with only horizontal viscosity and diffusivity (\ref{eq:ptl}) as the aspect ration parameter $\e$ goes to zero.

As mentioned in the introduction, for initial data $(v_0,\r_0) \in H^1(\O)$, the system (\ref{eq:ptv}) subject to (\ref{ga:are})-(\ref{eq:eve}) has a global weak solution $(v_\e,w_\e,\r_\e)$, while the system (\ref{eq:ptl}) corresponding to (\ref{ga:are})-(\ref{eq:eve}) exists a unique local strong solution $(v,\r)$.~Denote by $t^*_0$ the maximal existence time of this local strong solution.~The well-posedness of strong solutions to the primitive equations with only horizontal viscosity and diffusivity (\ref{eq:ptl}) is as follows (see \cite{cc2017}).

\begin{proposition}\label{po:lgc}
Let $v_0$,$\r_0\in H^1(\O)$ be two periodic functions with $\dz{\nh \d v_0(x,y,z)}=0,~\textnormal{for all}~(x,y)\in M$. Then the following assertions hold true:\\
(i) There exists a unique local strong solution $(v,\r)$ to the primitive equations with only horizontal viscosity and diffusivity (\ref{eq:ptl}) corresponding to (\ref{ga:are})-(\ref{eq:eve}), such that
\begin{gather*}
  (v,\r) \in L^{\infty}\xkh{[0,t^*_0);H^1(\O)} \cap C\xkh{[0,t^*_0);L^2(\O)},\\
  (\nh v,\nh \r) \in L^2\xkh{[0,t^*_0);H^1(\O)},(\p_t v,\p_t \r) \in L^2\xkh{[0,t^*_0);L^2(\O)},
\end{gather*}where $t^*_0$ is the maximal existence time of this local strong solution;\\
(ii) The local strong solution $(v,\r)$ to the system (\ref{eq:ptl}) satisfies the following energy estimate
\begin{flalign}
  &\sup_{0 \leq s \leq t}\xkh{\norm{(v,\r)}^2_{H^1(\O)}}(s)+\ds{\norm{\nh v}^2_{H^1(\O)}}\nonumber\\
  &\qquad+\ds{\xkh{{\norm{\nh \r}^2_{H^1(\O)}}+\norm{(\p_t v,\p_t \r)}^2_{L^2(\O)}}} \leq C,\label{fl:VTC}
\end{flalign}for any $t \in [0,t^*_0)$.~Here $C$ is a positive constant.
\end{proposition}

The following proposition is formally obtained by testing the scaled Boussinesq equations with rotation (\ref{eq:ptv}) with $(v,w,\r)$.~As for the rigorous justification for this proposition, we refer to the work of Li-Titi\cite{lt2019} and Bardos \et\cite{ba2013}.
\begin{proposition}\label{po:vew}
Given a periodic function pair $(v_0,\r_0) \in H^1(\O)$ with
\begin{equation*}
  \dz{\nh \d v_0}=0~\textnormal{and}~w_0(x,y,z)=-\dk{\nh \d v_0(x,y,\xi)}.
\end{equation*}Suppose that $(v_\e,w_\e,\r_\e)$ is a global weak solution of the system (\ref{eq:ptv}), satisfying the energy inequality (\ref{fl:vwt}), and that $(v,\r)$ is the unique local strong solution of the system (\ref{eq:ptl}). Then the following integral equality holds
\begin{flalign}
  &\xkh{\oo{\xkh{v_\e \d v+\e^2 w_\e w+\r_\e \r}}}(r)+\rt{f_0(\vec{k} \times v_\e) \d v}\nonumber\\
  &\qquad\qquad+\rt{\zkh{\nh v_\e:\nh v+\e^{\be-2}(\p_z v_\e) \d \p_z v+\e^2\nh w_\e \d \nh w}}\nonumber\\
  &\qquad\qquad+\rt{\zkh{\e^\be(\p_z w_\e)\p_z w+\nh \r_\e \d \nh \r+\e^{\ga-2}(\p_z\r_\e)\p_z\r}}\nonumber\\
  &\qquad=\norm{v_0}^2_2+\rt{\zkh{-(u_\e \d \n)v_\e \d v-\e^2(u_\e \d \n w_\e)w-(u_\e \d \n \r_\e)\r}}\nonumber\\
  &\qquad\qquad+\frac{\e^2}{2}\norm{w(r)}^2_2+\frac{\e^2}{2}\norm{w_0}^2_2+\e^2\rt{\xkh{\dk{\p_t v(x,y,\xi,t)}} \d \nh W_\e}\nonumber\\
  &\qquad\qquad+\norm{\r_0}^2_2+\rt{\xkh{v_\e \d \p_t v+\r_\e \p_t \r+\r_\e w-w_\e\r}}, \label{fl:rew}
\end{flalign}for any $r \in [0,t^*_0)$.
\end{proposition}

With the help of Proposition \ref{po:vew}, we can estimate the difference function $(V_\e,W_\e,\g_\e)=(v_\e-v,w_\e-w,\r_\e-\r)$. Before this, we present a lemma (see \cite{ct2003}), which will be frequently used later.
\begin{lemma} \label{le:phi}
The following inequalities hold
\begin{flalign*}
  \mm{&\xkh{\dz{\varphi(x,y,z)}} \xkh{\dz{\psi(x,y,z)\phi(x,y,z)}}}\\
  &\leq C\norm{\varphi}^{1/2}_2 \xkh{\norm{\varphi}^{1/2}_2+\norm{\nh\varphi}^{1/2}_2}
  \norm{\psi}^{1/2}_2 \xkh{\norm{\psi}^{1/2}_2+\norm{\nh\psi}^{1/2}_2} \norm{\phi}_2,
\end{flalign*}
\begin{flalign*}
  \mm{&\xkh{\dz{\varphi(x,y,z)}} \xkh{\dz{\psi(x,y,z)\phi(x,y,z)}}}\\
  &\leq C\norm{\psi}^{1/2}_2 \xkh{\norm{\psi}^{1/2}_2+\norm{\nh\psi}^{1/2}_2}
  \norm{\phi}^{1/2}_2 \xkh{\norm{\phi}^{1/2}_2+\norm{\nh\phi}^{1/2}_2} \norm{\varphi}_2,
\end{flalign*}for every~$\varphi,\psi,\phi$~such that the right-hand sides make sense and are finite,~where~$C$~is a positive constant.
\end{lemma}

\begin{proposition}\label{po:eta}
Let~$(V_\e,W_\e,\g_\e)=(v_\e-v,w_\e-w,\r_\e-\r)$.~Under the same assumptions as in Proposition \ref{po:vew}, the following estimate holds
\begin{flalign*}
  &\sup_{0 \leq s \leq t}\xkh{\norm{(V_\e,\e W_\e,\g_\e)}^2_2}(s)+\ds{\xkh{\norm{\nh V_\e}^2_2+\e^{\be-2}\norm{\p_z V_\e}^2_2}}\\
  &\qquad\quad+\ds{\xkh{\e^2\norm{\nh W_\e}^2_2+\norm{\nh \g_\e}^2_2+\e^\be\norm{\p_z W_\e}^2_2+\e^{\ga-2}\norm{\p_z \g_\e}^2_2}}\\
  &\qquad\leq C(t+1)e^{C(t+1)}\zkh{\e^2+\e^{\be-2}+\e^{\ga-2}+\e^2\xkh{\norm{v_0}^2_2+\e^2\norm{w_0}^2_2+\norm{\r_0}^2_2}^2},
\end{flalign*}for any $t \in [0,t^*_0)$, where $C$ is a positive constant that does not depend on $\e$.
\end{proposition}
\begin{proof}[Proof.]
Multiplying the first three equation in system (\ref{eq:ptl}) by $v_\e$, $w_\e$ and $\r_\e$ respectively,~and integrating over $\O\times(0,r)$, then it follows from integration by parts that
\begin{flalign}
  &\rt{\xkh{v_\e \d \p_t v+\r_\e \p_t \r+\nh v_\e:\nh v+\nh \r_\e \d \nh \r}}\nonumber\\
  &\qquad=\rt{\zkh{\r w_\e-w\r_\e-(u \d \n)v \d v_\e-(u \d \n \r)\r_\e}}\nonumber\\
  &\qquad\quad+\rt{\zkh{-f_0(\vec{k} \times v) \d v_\e}}.\label{fl:udn}
\end{flalign}Replacing $(v_\e,w_\e,\r_\e)$ with $(v,w,\r)$, a similar argument gives
\begin{flalign}
  &\frac{1}{2}\xkh{\norm{v(r)}^2_2+\norm{\r(r)}^2_2}+\dt{\xkh{\norm{\nh v}^2_2+\norm{\nh \r}^2_2}}\nonumber\\
  &\qquad=\frac{1}{2}\xkh{\norm{v_0}^2_2+\norm{\r_0}^2_2},\label{fl:rtw}
\end{flalign}note that we have used the following fact
\begin{equation*}
  \rt{f_0(\vec{k} \times v) \d v}=0.
\end{equation*}Thanks to Remark \ref{re:lh}, the weak solution $(v_\e,w_\e,\r_\e)$ of the system (\ref{eq:ptv}) satisfies the following energy inequality
\begin{flalign}
  &\frac{1}{2}\xkh{\norm{v_\e(r)}^2_2+\e^2\norm{w_\e(r)}^2_2+\norm{\r_\e(r)}^2_2}\nonumber\\
  &\qquad\quad+\dt{\xkh{\norm{\nh v_\e}^2_2+\e^{\be-2}\norm{\p_z v_\e}^2_2+\e^2\norm{\nh w_\e}^2_2}}\nonumber\\
  &\qquad\quad+\dt{\xkh{\e^\be\norm{\p_z w_\e}^2_2+\norm{\nh \r_\e}^2_2+\e^{\ga-2}\norm{\p_z \r_\e}^2_2}}\nonumber\\
  &\qquad\leq\frac{1}{2}\xkh{\norm{v_0}^2_2+\e^2\norm{w_0}^2_2+\norm{\r_0}^2_2}.\label{fl:e22}
\end{flalign}Subtracting the sum of (\ref{fl:rew}) and (\ref{fl:udn}) from the sum of (\ref{fl:rtw}) and (\ref{fl:e22}), we have
\begin{flalign}
  &\frac{1}{2}\xkh{\norm{V_\e(r)}^2_2+\e^2\norm{W_\e(r)}^2_2+\norm{\g_\e(r)}^2_2}
  +\dt{\xkh{\norm{\nh V_\e}^2_2+\e^{\be-2}\norm{\p_z V_\e}^2_2}}\nonumber\\
  &\qquad\quad+\dt{\xkh{\e^2\norm{\nh W_\e}^2_2+\norm{\nh \g_\e}^2_2
  +\e^\be\norm{\p_z W_\e}^2_2+\e^{\ga-2}\norm{\p_z \g_\e}^2_2}}\nonumber\\
  &\qquad\leq\rt{\zkh{(u_\e \d \n\r_\e)\r+(u \d \n\r)\r_\e}}\nonumber\\
  &\qquad\quad+\rt{\zkh{(u_\e \d \n)v_\e \d v+(u \d \n)v \d v_\e}}\nonumber\\
  &\qquad\quad+\e^2\rt{\zkh{-\xkh{\dk{\p_t v(x,y,\xi,t)}} \d \nh W_\e}}\nonumber\\
  &\qquad\quad+\rt{\zkh{-\e^2\nh W_\e \d \nh w-\e^{\ga-2}(\p_z\g_\e)\p_z \r}}\nonumber\\
  &\qquad\quad+\rt{\zkh{-\e^\be(\p_z W_\e)\p_z w-\e^{\be-2}(\p_z V_\e)\d\p_z v}}\nonumber\\
  &\qquad\quad+\e^2\rt{(u_\e \d \n w_\e)w}\nonumber\\
  &\qquad=:\ma{R}_1+\ma{R}_2+\ma{R}_3+\ma{R}_4+\ma{R}_5+\ma{R}_6.\label{fl:Ver}
\end{flalign}In order to estimate the first integral term $\mathcal{R}_1$ on the right-hand side of (\ref{fl:Ver}), we use the divergence-free condition and integration by parts to obtain
\begin{flalign}
  \ma{R}_1:&=\rt{\zkh{(u_\e \d \n\r_\e)\r+(u \d \n\r)\r_\e}}\nonumber\\
  &=\rt{\zkh{(u_\e \d \n\r_\e)\r-(u \d \n \r_\e)\r}}\nonumber\\
  &=\rt{\zkh{\xkh{(u_\e-u) \d \n\g_\e}\r}}\nonumber\\
  &=\rt{\zkh{\xkh{V_\e \d \nh\g_\e}\r+W_\e \xkh{\p_z \g_\e}\r}}\nonumber\\
  &=:\mathcal{R}_{11}+\mathcal{R}_{12}.\label{fl:ued}
\end{flalign}For the integral term $\ma{R}_{11}$ on the right-hand side of (\ref{fl:ued}), noting that the fact that $|\r(z)| \leq \frac{1}{2}\dz{|\r|}+\dz{|\p_{z}\r|}$, and applying Lemma \ref{le:phi} and Young inequality, we have
\begin{flalign}
  \ma{R}_{11}:&=\rt{\xkh{V_\e \d \nh\g_\e}\r} \leq \rt{|\r||V_\e||\nh\g_\e|}\nonumber\\
  &\leq \dt{\mm{\xkh{\dz{(|\r|+|\p_{z}\r|)}}\xkh{\dz{|V_\e||\nh\g_\e|}}}}\nonumber\\
  &\leq C\dt{\xkh{\norm{\p_z\r}_2+\norm{\p_z\r}^{1/2}_2\norm{\nh\p_z\r}^{1/2}_2}
  \norm{V_\e}^{1/2}_2\norm{\nh V_\e}^{1/2}_2\norm{\nh \g_\e}_2}\nonumber\\
  &\quad+C\dt{\norm{V_\e}_2\norm{\nh \g_\e}_2\xkh{\norm{\p_z\r}_2+\norm{\p_z\r}^{1/2}_2\norm{\nh\p_z\r}^{1/2}_2}}\nonumber\\
  &\quad+C\dt{\xkh{\norm{\r}_2+\norm{\r}^{1/2}_2\norm{\nh\r}^{1/2}_2}
  \xkh{\norm{V_\e}_2+\norm{V_\e}^{1/2}_2\norm{\nh V_\e}^{1/2}_2}\norm{\nh \g_\e}_2}\nonumber\\
  &\leq C\dt{\xkh{\norm{\p_z\r}^2_2+\norm{\p_z\r}_2\norm{\nh\p_z\r}_2+\norm{\p_z\r}^4_2
  +\norm{\p_z\r}^2_2\norm{\nh\p_z\r}^2_2}\norm{V_\e}^2_2}\nonumber\\
  &\quad+C\dt{\xkh{\norm{\r}^2_2+\norm{\r}_2\norm{\nh\r}_2+\norm{\r}^4_2
  +\norm{\r}^2_2\norm{\nh\r}^2_2}\norm{V_\e}^2_2}\nonumber\\
  &\quad+\frac{1}{24}\dt{\xkh{\norm{\nh V_\e}^2_2+\norm{\nh \g_\e}^2_2}}.\label{fl:R11}
\end{flalign}Next we need to estimate the integral term $\ma{R}_{12}$ on the right-hand side of (\ref{fl:ued}). Using the same method as the first integral term on the right-hand side of (\ref{fl:ued}), this term can be bounded as
\begin{flalign}
  \ma{R}_{12}:&=\rt{W_\e \xkh{\p_z \g_\e}\r}\nonumber\\
  &=\rt{\zkh{(\nh \d V_\e)\g_\e \r+\g_\e (\p_z \r)\dk{\nh \d V_\e}}}\nonumber\\
  &\leq \dt{\mm{\xkh{\dz{(|\r|+|\p_{z}\r|)}}\xkh{\dz{|\g_\e||\nh V_\e|}}}}\nonumber\\
  &\quad+\dt{\mm{\xkh{\dz{|\nh V_\e|}}\xkh{\dz{|\g_\e||\p_z \r|}}}}\nonumber\\
  &\leq C\dt{\xkh{\norm{\p_z\r}^2_2+\norm{\p_z\r}_2\norm{\nh\p_z\r}_2+\norm{\p_z\r}^4_2
  +\norm{\p_z\r}^2_2\norm{\nh\p_z\r}^2_2}\norm{\g_\e}^2_2}\nonumber\\
  &\quad+C\dt{\xkh{\norm{\r}^2_2+\norm{\r}_2\norm{\nh\r}_2+\norm{\r}^4_2
  +\norm{\r}^2_2\norm{\nh\r}^2_2}\norm{\g_\e}^2_2}\nonumber\\
  &\quad+\frac{1}{24}\dt{\xkh{\norm{\nh V_\e}^2_2+\norm{\nh \g_\e}^2_2}}.\label{fl:R12}
\end{flalign}Adding (\ref{fl:R11}) and (\ref{fl:R12}) gives
\begin{flalign*}
  \ma{R}_1&\leq C\dt{\xkh{\norm{\p_z\r}^2_2+\norm{\p_z\r}_2\norm{\nh\p_z\r}_2}\xkh{\norm{V_\e}^2_2+\norm{\g_\e}^2_2}}\\
  &\quad+C\dt{\xkh{\norm{\p_z\r}^4_2+\norm{\p_z\r}^2_2\norm{\nh\p_z\r}^2_2}\xkh{\norm{V_\e}^2_2+\norm{\g_\e}^2_2}}\\
  &\quad+C\dt{\xkh{\norm{\r}^2_2+\norm{\r}_2\norm{\nh\r}_2}\xkh{\norm{V_\e}^2_2+\norm{\g_\e}^2_2}}\\
  &\quad+C\dt{\xkh{\norm{\r}^4_2+\norm{\r}^2_2\norm{\nh\r}^2_2}\xkh{\norm{V_\e}^2_2+\norm{\g_\e}^2_2}}\\
  &\quad+\frac{1}{12}\dt{\xkh{\norm{\nh V_\e}^2_2+\norm{\nh \g_\e}^2_2}}.
\end{flalign*}The similar argument as the integral term $\ma{R}_1$ yields
\begin{flalign*}
  \ma{R}_2:&=\rt{\zkh{(u_\e \d \n)v_\e \d v+(u \d \n)v \d v_\e}}\\
  &=\rt{\zkh{(V_\e \d \nh)V_\e \d v+(\nh \d V_\e)V_\e \d v+(V_\e \d \p_z v)\dk{\nh \d V_\e}}}\\
  &\leq C\dt{\xkh{\norm{v}^2_2+\norm{v}_2\norm{\nh v}_2+\norm{v}^4_2+\norm{v}^2_2\norm{\nh v}^2_2}\norm{V_\e}^2_2}
  +\frac{1}{12}\dt{\norm{\nh V_\e}^2_2}\\
  &\quad+C\dt{\xkh{\norm{\p_z v}^2_2+\norm{\p_z v}_2\norm{\nh \p_z v}_2
  +\norm{\p_z v}^4_2+\norm{\p_z v}^2_2\norm{\nh \p_z v}^2_2}\norm{V_\e}^2_2}.
\end{flalign*}By virtue of the H\"{o}lder inequality and Young inequality, the integral terms $\ma{R}_3$, $\ma{R}_4$ and $\ma{R}_5$ on the right-hand side of (\ref{fl:Ver}) can be estimated as
\begin{flalign*}
  \ma{R}_3:&=\e^2\rt{\zkh{-\xkh{\dk{\p_t v(x,y,\xi,t)}} \d \nh W_\e}}\\
  &\leq\e^2\dt{\mm{\xkh{\dz{|\p_t v|}}\xkh{\dz{|\nh W_\e|}}}}\\
  &\leq C\e^2\dt{\norm{\p_t v}^2_2}+\frac{1}{12}\dt{\e^2\norm{\nh W_\e}^2_2},
\end{flalign*}
\begin{flalign*}
  \ma{R}_4:&=\rt{\zkh{-\e^2 \nh W_\e\d\nh w-\e^{\ga-2}(\p_z \g_\e)\p_z T}}\\
  &\leq C\e^2\dt{\norm{\nh w}^2_2}+\frac{1}{12}\dt{\e^2\norm{\nh W_\e}^2_2}\\
  &\quad+C\e^{\ga-2}\dt{\norm{\p_z T}^2_2}+\frac{1}{12}\dt{\e^{\ga-2}\norm{\p_z \g_\e}^2_2}\\
  &\leq C\e^2\rt{|\nh(\nh \d v)|^2}+C\e^{\ga-2}\dt{\norm{\p_z T}^2_2}\\
  &\quad+\frac{1}{12}\dt{\xkh{\e^2\norm{\nh W_\e}^2_2+\e^{\ga-2}\norm{\p_z \g_\e}^2_2}}\\
  &\leq C\e^2\dt{\norm{\nh^2 v}^2_2}+C\e^{\ga-2}\dt{\norm{\p_z T}^2_2}\\
  &\quad+\frac{1}{12}\dt{\xkh{\e^2\norm{\nh W_\e}^2_2+\e^{\ga-2}\norm{\p_z \g_\e}^2_2}}
\end{flalign*}and
\begin{flalign*}
  \ma{R}_5:&=\rt{\zkh{-\e^\be(\p_z W_\e)\p_z w-\e^{\be-2}(\p_z V_\e)\d\p_z v}}\\
  &\leq C\e^\be\dt{\norm{\p_z w}^2_2}+\frac{1}{12}\dt{\e^\be\norm{\p_z W_\e}^2_2}\\
  &\quad+C\e^{\be-2}\dt{\norm{\p_z v}^2_2}+\frac{1}{12}\dt{\e^{\be-2}\norm{\p_z V_\e}^2_2}\\
  &\leq C\e^\be\dt{\norm{\nh v}^2_2}+C\e^{\be-2}\dt{\norm{\n v}^2_2}\\
  &\quad+\frac{1}{12}\dt{\xkh{\e^\be\norm{\p_z W_\e}^2_2+\e^{\be-2}\norm{\p_z V_\e}^2_2}},
\end{flalign*}respectively, noting that the divergence-free condition is used. Finally, it remains to deal with the last integral term $\ma{R}_6$ on the right-hand side of (\ref{fl:Ver}). Thanks to the Lemma \ref{le:phi} and Young inequality, we obtain
\begin{flalign*}
  \ma{R}_6:&=\e^2\rt{(u_\e \d \n w_\e)w}\\
  &=\e^2\rt{\zkh{w_\e(\nh \d V_\e)\dk{\nh \d v}-v_\e \d \nh W_\e\dk{\nh \d v}}}\\
  &\leq\e^2\dt{\mm{\xkh{\dz{|\nh v|}}\xkh{\dz{|w_\e||\nh V_\e|}}}}\\
  &\quad+\e^2\dt{\mm{\xkh{\dz{|\nh v|}}\xkh{\dz{|v_\e||\nh W_\e|}}}}\\
  &\leq C\e^2\norm{\nh v}^{1/2}_2\norm{\nh^2 v}^{1/2}_2\norm{w_\e}^{1/2}_2
  \xkh{\norm{w_\e}^{1/2}_2+\norm{\nh w_\e}^{1/2}_2}\norm{\nh V_\e}_2\\
  &\quad+C\e^2\norm{\nh v}^{1/2}_2\norm{\nh^2 v}^{1/2}_2\norm{v_\e}^{1/2}_2
  \xkh{\norm{v_\e}^{1/2}_2+\norm{\nh v_\e}^{1/2}_2}\norm{\nh W_\e}_2\\
  &\leq C\e^2\dt{\xkh{\norm{\nh v}^2_2\norm{\nh^2 v}^2_2
  +\e^4\norm{w_\e}^4_2+\e^4\norm{w_\e}^2_2\norm{\nh w_\e}^2_2}}\\
  &\quad+C\e^2\dt{\xkh{\norm{v_\e}^4_2+\norm{v_\e}^2_2\norm{\nh v_\e}^2_2}}
  +\frac{1}{12}\dt{\xkh{\norm{\nh V_\e}^2_2+\e^2\norm{\nh W_\e}^2_2}}.
\end{flalign*}Combining the estimates for $\ma{R}_1$, $\ma{R}_2$, $\ma{R}_3$, $\ma{R}_4$, $\ma{R}_5$ and $\ma{R}_6$, we reach
\begin{flalign*}
  \ma{F}(t):&=\xkh{\norm{V_\e(t)}^2_2+\e^2\norm{W_\e(t)}^2_2+\norm{\g_\e(t)}^2_2}+\ds{\xkh{\norm{\nh V_\e}^2_2+\e^{\be-2}\norm{\p_z V_\e}^2_2}}\\
  &\quad+\ds{\xkh{\e^2\norm{\nh W_\e}^2_2+\norm{\nh \g_\e}^2_2+\e^\be\norm{\p_z W_\e}^2_2+\e^{\ga-2}\norm{\p_z \g_\e}^2_2}}\\
  &\leq C\ds{\xkh{\norm{\p_z\r}^2_2+\norm{\p_z\r}_2\norm{\nh\p_z\r}_2+\norm{\p_z\r}^4_2+\norm{\p_z\r}^2_2\norm{\nh\p_z\r}^2_2}\norm{V_\e}^2_2}\\
  &\quad+C\ds{\xkh{\norm{\p_z\r}^2_2+\norm{\p_z\r}_2\norm{\nh\p_z\r}_2+\norm{\p_z\r}^4_2+\norm{\p_z\r}^2_2\norm{\nh\p_z\r}^2_2}\norm{\g_\e}^2_2}\\
  &\quad+C\ds{\xkh{\norm{\r}^2_2+\norm{\r}_2\norm{\nh\r}_2+\norm{\r}^4_2+\norm{\r}^2_2\norm{\nh\r}^2_2}\xkh{\norm{V_\e}^2_2+\norm{\g_\e}^2_2}}\\
  &\quad+C\ds{\xkh{\norm{\p_z v}^2_2+\norm{\p_z v}_2\norm{\nh\p_z v}_2+\norm{\p_z v}^4_2+\norm{\p_z v}^2_2\norm{\nh\p_z v}^2_2}\norm{V_\e}^2_2}\\
  &\quad+C\ds{\xkh{\norm{v}^2_2+\norm{v}_2\norm{\nh v}_2+\norm{v}^4_2+\norm{v}^2_2\norm{\nh v}^2_2}\norm{V_\e}^2_2}\\
  &\quad+C\e^2\ds{\xkh{\norm{\p_t v}^2_2+\norm{\nh^2 v}^2_2}}+C\e^{\ga-2}\ds{\norm{\p_z \r}^2_2}+C\e^\be\ds{\norm{\nh v}^2_2}\\
  &\quad+C\e^2\ds{\xkh{\norm{\nh v}^2_2\norm{\nh^2 v}^2_2+\e^4\norm{w_\e}^4_2+\e^4\norm{w_\e}^2_2\norm{\nh w_\e}^2_2}}\\
  &\quad+C\e^2\ds{\xkh{\norm{v_\e}^4_2+\norm{v_\e}^2_2\norm{\nh v_\e}^2_2}}+C\e^{\be-2}\ds{\norm{\n v}^2_2}=:\ma{G}(t),
\end{flalign*}for a.e. $t \in [0,t^*_0)$. Taking the derivative of $\ma{G}(t)$ with respect to $t$ leads to
\begin{flalign*}
  \ma{G}'(t)&=C\xkh{\norm{\p_z\r}^2_2+\norm{\p_z\r}_2\norm{\nh\p_z\r}_2+\norm{\p_z\r}^4_2+\norm{\p_z\r}^2_2\norm{\nh\p_z\r}^2_2}\xkh{\norm{V_\e}^2_2+\norm{\g_\e}^2_2}\\
  &\quad+C\xkh{\norm{\r}^2_2+\norm{\r}_2\norm{\nh\r}_2+\norm{\r}^4_2+\norm{\r}^2_2\norm{\nh\r}^2_2}\xkh{\norm{V_\e}^2_2+\norm{\g_\e}^2_2}\\
  &\quad+C\xkh{\norm{\p_z v}^2_2+\norm{\p_z v}_2\norm{\nh\p_z v}_2+\norm{\p_z v}^4_2+\norm{\p_z v}^2_2\norm{\nh\p_z v}^2_2}\norm{V_\e}^2_2\\
  &\quad+C\xkh{\norm{v}^2_2+\norm{v}_2\norm{\nh v}_2+\norm{v}^4_2+\norm{v}^2_2\norm{\nh v}^2_2}\norm{V_\e}^2_2\\
  &\quad+C\e^2\xkh{\norm{\p_t v}^2_2+\norm{\nh^2 v}^2_2}+C\e^{\ga-2}\norm{\p_z \r}^2_2+C\e^\be\norm{\nh v}^2_2\\
  &\quad+C\e^2\xkh{\norm{\nh v}^2_2\norm{\nh^2 v}^2_2+\e^4\norm{w_\e}^4_2+\e^4\norm{w_\e}^2_2\norm{\nh w_\e}^2_2}\\
  &\quad+C\e^2\xkh{\norm{v_\e}^4_2+\norm{v_\e}^2_2\norm{\nh v_\e}^2_2}+C\e^{\be-2}\norm{\n v}^2_2\\
  &\leq \Big[C\xkh{\norm{\p_z\r}^2_2+\norm{\p_z\r}_2\norm{\nh\p_z\r}_2+\norm{\p_z\r}^4_2+\norm{\p_z\r}^2_2\norm{\nh\p_z\r}^2_2}\\
  &\quad+C\xkh{\norm{\r}^2_2+\norm{\r}_2\norm{\nh\r}_2+\norm{\r}^4_2+\norm{\r}^2_2\norm{\nh\r}^2_2}\\
  &\quad+C\xkh{\norm{\p_z v}^2_2+\norm{\p_z v}_2\norm{\nh\p_z v}_2+\norm{\p_z v}^4_2+\norm{\p_z v}^2_2\norm{\nh\p_z v}^2_2}\\
  &\quad+C\xkh{\norm{v}^2_2+\norm{v}_2\norm{\nh v}_2+\norm{v}^4_2+\norm{v}^2_2\norm{\nh v}^2_2}\Big]\ma{G}(t)\\
  &\quad+C\e^2\xkh{\norm{\p_t v}^2_2+\norm{\nh^2 v}^2_2}+C\e^{\ga-2}\norm{\p_z \r}^2_2+C\e^\be\norm{\nh v}^2_2\\
  &\quad+C\e^2\xkh{\norm{\nh v}^2_2\norm{\nh^2 v}^2_2+\e^4\norm{w_\e}^4_2+\e^4\norm{w_\e}^2_2\norm{\nh w_\e}^2_2}\\
  &\quad+C\e^2\xkh{\norm{v_\e}^4_2+\norm{v_\e}^2_2\norm{\nh v_\e}^2_2}+C\e^{\be-2}\norm{\n v}^2_2,
\end{flalign*}where we have used the inequality $\ma{F}(t) \leq \ma{G}(t)$. Noting that the fact that $\mathcal{G}(0)=0$, and applying the Gronwall inequality to the above inequality, we obtain
\begin{flalign}
  \ma{F}(t)&\leq\exp\Bigg\{C\ds{\xkh{\norm{\p_z\r}^2_2+\norm{\p_z\r}_2\norm{\nh\p_z\r}_2+\norm{\p_z\r}^4_2+\norm{\p_z\r}^2_2\norm{\nh\p_z\r}^2_2}}\nonumber\\
  &\quad+C\ds{\xkh{\norm{\r}^2_2+\norm{\r}_2\norm{\nh\r}_2+\norm{\r}^4_2+\norm{\r}^2_2\norm{\nh\r}^2_2}}\nonumber\\
  &\quad+C\ds{\xkh{\norm{\p_z v}^2_2+\norm{\p_z v}_2\norm{\nh\p_z v}_2+\norm{\p_z v}^4_2+\norm{\p_z v}^2_2\norm{\nh\p_z v}^2_2}}\nonumber\\
  &\quad+C\ds{\xkh{\norm{v}^2_2+\norm{v}_2\norm{\nh v}_2+\norm{v}^4_2+\norm{v}^2_2\norm{\nh v}^2_2}}\Bigg\}\nonumber\\
  &\times\Bigg\{C\e^2\ds{\xkh{\norm{\p_t v}^2_2+\norm{\nh^2 v}^2_2}}+C\e^{\ga-2}\ds{\norm{\p_z \r}^2_2}+C\e^\be\ds{\norm{\nh v}^2_2}\nonumber\\
  &\quad+C\e^2\ds{\xkh{\norm{\nh v}^2_2\norm{\nh^2 v}^2_2+\e^4\norm{w_\e}^4_2+\e^4\norm{w_\e}^2_2\norm{\nh w_\e}^2_2}}\nonumber\\
  &\quad+C\e^2\ds{\xkh{\norm{v_\e}^4_2+\norm{v_\e}^2_2\norm{\nh v_\e}^2_2}}+C\e^{\be-2}\ds{\norm{\n v}^2_2}\Bigg\}.\label{fl:Ft}
\end{flalign}From (\ref{fl:VTC}) and (\ref{fl:e22}), it follows that
\begin{flalign*}
  &\xkh{\norm{V_\e(t)}^2_2+\e^2\norm{W_\e(t)}^2_2+\norm{\g_\e(t)}^2_2}+\ds{\xkh{\norm{\nh V_\e}^2_2+\e^{\be-2}\norm{\p_z V_\e}^2_2}}\\
  &\qquad\quad+\ds{\xkh{\e^2\norm{\nh W_\e}^2_2+\norm{\nh \g_\e}^2_2+\e^\be\norm{\p_z W_\e}^2_2+\e^{\ga-2}\norm{\p_z \g_\e}^2_2}}\\
  &\qquad\leq C(t+1)e^{C(t+1)}\zkh{\e^2+\e^{\be-2}+\e^{\ga-2}+\e^2\xkh{\norm{v_0}^2_2+\e^2\norm{w_0}^2_2+\norm{\r_0}^2_2}^2}.
\end{flalign*}This completes the proof of Proposition \ref{po:eta}.
\end{proof}

Based on Proposition \ref{po:eta}, we give the Proof of Theorem~\ref{th:VTC}.
\begin{proof}[Proof of Theorem~\ref{th:VTC}.]
By the Proposition \ref{po:eta}, we have the following estimate
\begin{flalign*}
  &\sup_{0 \leq t < t^*_0}\xkh{\norm{(V_\e,\e W_\e,\g_\e)}^2_2}(t)+\dto{\xkh{\norm{\nh V_\e}^2_2+\e^{\be-2}\norm{\p_z V_\e}^2_2}}\\
  &\qquad\quad+\dto{\xkh{\e^2\norm{\nh W_\e}^2_2+\norm{\nh \g_\e}^2_2+\e^\be\norm{\p_z W_\e}^2_2+\e^{\ga-2}\norm{\p_z \g_\e}^2_2}}\\
  &\qquad\leq C\e^\eta(t^*_0+1)e^{C(t^*_0+1)}\zkh{1+\xkh{\norm{v_0}^2_2+\norm{w_0}^2_2+\norm{\r_0}^2_2}^2},
\end{flalign*}where $\eta=\min\{2,\be-2,\ga-2\}$ with $2<\be,\ga<\infty$, and $t^*_0$ is the maximal existence time of local strong solution $(v,\r)$ to the system (\ref{eq:ptl}). Here $C$ is a positive constant that does not depend on $\e$. The above estimate implies that
\begin{gather*}
  (v_\e,\e w_\e,\r_\e) \rightarrow (v,0,\r),~in~L^{\infty}\xkh{[0,t^*_0);L^2(\O)},\\
  \xkh{\nh v_\e,\e^{(\be-2)/2}\p_z v_\e,\e \nh w_\e} \rightarrow (\nh v,0,0),~in~L^2\xkh{[0,t^*_0);L^2(\O)},\\
  \xkh{\nh \r_\e,\e^{\be/2}\p_z w_\e,\e^{(\ga-2)/2}\p_z \r_\e} \rightarrow (\nh \r,0,0),~in~L^2\xkh{[0,t^*_0);L^2(\O)}.
\end{gather*}Owing to $\nh v_\e \rightarrow \nh v~\textnormal{in}~L^2\xkh{[0,t^*_0);L^2(\O)}$, it deduces from the divergence-free condition that
\begin{equation*}
  w_\e \rightarrow w~in~L^2\xkh{[0,t^*_0);L^2(\O)}.
\end{equation*}Finally, it can easily be seen that the rate of convergence is of the order $O(\e^{\eta/2})$.
\end{proof}

\section{Strong convergence for $H^1$ initial data with additional regularity}
In this section, we study the strong convergence for the case of initial data $(v_0,\r_0) \in H^1(\O)$ with additional regularity $\p_z v_0 \in L^4(\O)$.~The following global existence result is due to Cao-Li-Titi\cite{cc2017}.
\begin{proposition}\label{po:J1t}
Assume that $(v_0,\r_0) \in H^1(\O) \cap L^\infty(\O)$ satisfying $\dz{\nh \d v_0}=0$, and that $\p_z v_0 \in L^q(\O)$ with $q \in (2,\infty)$. Then the local strong solution $(v,\r)$ of the system (\ref{eq:ptl}) subject to (1.5)-(1.7) can be extended uniquely to be a global one such that the following energy estimate holds
\begin{flalign}
  &\sup_{0 \leq s \leq t}\xkh{\norm{(v,\r)}^2_{H^1(\O)}}(s)+\ds{\norm{\nh v}^2_{H^1(\O)}}\nonumber\\
  &\qquad+\ds{\xkh{{\norm{\nh \r}^2_{H^1(\O)}}+\norm{(\p_t v,\p_t \r)}^2_{L^2(\O)}}} \leq \ma{J}_1(t),\label{fl:J1t}
\end{flalign}for any $t \in [0,\infty)$, where $\ma{J}_1(t)$ is a nonnegative continuously increasing function defined on $[0,\infty)$.
\end{proposition}

In order to prove the Theorem~\ref{th:J2t}, we need the following proposition, which is a direct consequence of Lemma 2.2 with exponent $(6,6,2)$ in \cite{yu2022}.
\begin{proposition}\label{po:delta}
Let $(v,\r)$ be the local strong solution to the system (\ref{eq:ptl}) corresponding to (1.5)-(1.7). Then the following inequalities hold
\begin{flalign*}
  &\sup_{(x,y,z) \in \overline{\O}}|v(x,y,z,t_\delta)|+\sup_{(x,y,z) \in \overline{\O}}|\r(x,y,z,t_\delta)|\\
  &\qquad\leq C\xkh{\norm{\nh v}_6+\norm{\p_z v}_2+\norm{\nh \r}_6+\norm{\p_z \r}_2}(t_\delta)
\end{flalign*}
for some fixed time $t_\delta \in [0,t^*_0)$.
\end{proposition}

With the help of Proposition \ref{po:J1t} and \ref{po:delta}, we give the proof of Theorem~\ref{th:J2t}.~The key to the proof is to establish $L^\infty([0,t^*_0);L^4({\O}))$ estimate on the vertical derivative of horizontal velocity that does not depend on $\norm{(v_0,\r_0)}_\infty$.

\begin{proof}[Proof of Theorem~\ref{th:J2t}.]
Thanks to the energy estimate (\ref{fl:VTC}), we obtain
\begin{equation*}
  \int^{t^*_0}_{t^*_0/4}\xkh{\norm{\n\nh v}^2_2+\norm{\n\nh \r}^2_2}ds \leq C,
\end{equation*}which implies that there exists a fixed time $t_0 \in (t^*_0/4,t^*_0)$ such that
\begin{equation}\label{eq:ct0}
  \norm{\n\nh v}^2_2(t_0)+\norm{\n\nh \r}^2_2(t_0) \leq C/t^*_0.
\end{equation}Moreover, the following the energy estimate holds
\begin{flalign}
  &\sup_{0 \leq s \leq t_0}\xkh{\norm{(v,\r)}^2_{H^1(\O)}}(s)+\int^{t_0}_0{\norm{\nh v}^2_{H^1(\O)}}ds\nonumber\\
  &\qquad+\int^{t_0}_0{\xkh{{\norm{\nh \r}^2_{H^1(\O)}}+\norm{(\p_t v,\p_t \r)}^2_{L^2(\O)}}}ds \leq C.\label{fl:0t0}
\end{flalign}According to Proposition \ref{po:delta}, it deduces from (\ref{eq:ct0}), (\ref{fl:0t0}) and Sobolev imbedding theorem that
\begin{flalign*}
  &\sup_{(x,y,z) \in \overline{\O}}|v(x,y,z,t_0)|+\sup_{(x,y,z) \in \overline{\O}}|\r(x,y,z,t_0)|\\
  &\qquad\leq C\xkh{\norm{\n\nh v}_2(t_0)+\norm{\nh v}_2(t_0)+\norm{\p_z v}_2(t_0)}\\
  &\qquad\quad+C\xkh{\norm{\n\nh \r}_2(t_0)+\norm{\nh \r}_2(t_0)+\norm{\p_z \r}_2(t_0)}\\
  &\qquad\leq C\xkh{\norm{\n\nh v}^2_2(t_0)+\norm{\n v}^2_2(t_0)+\norm{\n\nh \r}^2_2(t_0)+\norm{\n \r}^2_2(t_0)+1}\\
  &\qquad\leq C(1+1/t^*_0).
\end{flalign*}The above inequality leads to $(v(x,y,z,t_0),\r(x,y,z,t_0)) \in L^\infty(\O)$.

Next, we show that $\p_z v \in L^\infty([0,t^*_0);L^4({\O}))$. Integrating the second equation of the system (\ref{eq:ptl}) with respect to $z$ gives
\begin{equation*}
  p(x,y,z,t)=p_\nu(x,y,t)+\dk{\r(x,y,\xi,t)},
\end{equation*}where $p_\nu(x,y,t)$ represents unknown surface pressure as $z=0$. Based on the above relation, we can rewrite the first equation of the system (\ref{eq:ptl}) as
\begin{flalign*}
  &\p_t v-\la_h v+(v \d \nh)v-\xkh{\dk{\nh \d v(x,y,\xi,t)}} \p_z v+\nh p_\nu(x,y,t)\\
  &+\dk{\nh \r(x,y,\xi,t)}+f_0\vec{k}\times v=0,
\end{flalign*}note that the divergence-free condition is used. Differentiating the above equation with respect to $z$, we have
\begin{flalign}
  &\p_t (\p_z v)-\la_h (\p_z v)+(\p_z v \d \nh)v-(\nh \d v)\p_z v+\nh\r+f_0\vec{k}\times(\p_z v)\nonumber\\
  &+(v \d \nh)\p_z v-\xkh{\dk{\nh \d v(x,y,\xi,t)}}\p_{zz} v=0.\label{fl:pzv}
\end{flalign}Taking the $L^2(\O)$ inner product of the equation (\ref{fl:pzv}) with $|\p_z v|^2\p_z v$, then it follows from integration by parts that
\begin{flalign}
  \frac{1}{4}\frac{d}{dt}&\norm{\p_z v}^4_4+\oo{\xkh{|\p_z v|^2|\nh \p_z v|^2+2|\p_z v|^2|\nh|\p_z v||^2}}\nonumber\\
  &=\oo{\zkh{(\nh \d v)\p_z v-(\p_z v \d \nh)v} \d |\p_z v|^2\p_z v}\nonumber\\
  &\quad+\oo{\zkh{\r |\p_z v|^2 (\nh \d \p_z v)+\r(\p_z v) \d \xkh{\nh|\p_z v|^2}}},\label{fl:v44}
\end{flalign}where we have used the following facts that
\begin{gather*}
  \oo{\zkh{\xkh{v \d \nh}\p_z v-\xkh{\dk{\nh \d v(x,y,\xi,t)}}\p_{zz} v} \d |\p_z v|^2 \p_z v}=0,\\
  \oo{\zkh{f_0\vec{k}\times(\p_z v)} \d |\p_z v|^2\p_z v}=0.
\end{gather*}For the first integral term on the right-hand side of (\ref{fl:v44}), using the Lemma \ref{le:phi} and Young inequality yields
\begin{flalign}
  &\oo{\zkh{(\nh \d v)\p_z v-(\p_z v \d \nh)v} \d |\p_z v|^2\p_z v}\nonumber\\
  &\qquad\leq C\oo{|\nh v||\p_z v|^2|\p_z v|^2}\nonumber\\
  &\qquad\leq C\mm{\xkh{\dz{(|\nh v|+|\nh \p_z v|)}}\xkh{\dz{|\p_z v|^2|\p_z v^2|}}}\nonumber\\
  &\qquad\leq C\xkh{\oo{|\p_z v|^4}}(\norm{\nh v}_2+\norm{\nh\p_z v}_2)\nonumber\\
  &\qquad\quad+C(\norm{\nh v}_2+\norm{\nh\p_z v}_2)\norm{\p_z v}^2_4\xkh{\oo{|\p_z v|^2|\nh \p_z v|^2}}^{1/2}\nonumber\\
  &\qquad\leq C\xkh{1+\norm{\nh v}^2_{H^1}}\norm{\p_z v}^4_4+\frac{3}{8}\oo{|\p_z v|^2|\nh \p_z v|^2},\label{fl:p41}
\end{flalign}note that the boundary condition (\ref{ga:are}) and symmetry condition (\ref{eq:eve}) are used. Applying the H\"{o}lder inequality, Young inequality and Sobolev imbedding theorem, the last integral term on the right-hand side of (\ref{fl:v44}) can be bounded as
\begin{flalign}
  &\oo{\zkh{\r |\p_z v|^2 (\nh \d \p_z v)+\r(\p_z v) \d \xkh{\nh|\p_z v|^2}}}\nonumber\\
  &\qquad\leq C\oo{|\r||\p_z v|^2|\nh \p_z v|}\nonumber\\
  &\qquad\leq C\oo{|\r|^2|\p_z v|^2}+\frac{3}{8}\oo{|\p_z v|^2|\nh \p_z v|^2}\nonumber\\
  &\qquad\leq C\norm{\r}^2_{H^1}\xkh{1+\norm{\p_z v}^4_4}+\frac{3}{8}\oo{|\p_z v|^2|\nh \p_z v|^2}.\label{fl:p42}
\end{flalign}Adding (\ref{fl:p41}) and (\ref{fl:p42}), then it deduces from the Gronwall inequality and Proposition \ref{po:lgc} that
\begin{flalign*}
  &\sup_{0 \leq s < t^*_0}\norm{\p_z v}_4(s)\leq\exp\dkh{C\int^{t^*_0}_0\xkh{1+\norm{\nh v}^2_{H^1}+\norm{\r}^2_{H^1}}ds}\\
  &\qquad\times\zkh{\norm{\p_z v_0}^4_4+C\int^{t^*_0}_0\norm{\r}^2_{H^1}ds}^{1/4} \leq Ce^{C(1+t^*_0)}\xkh{t^*_0+\norm{\p_z v_0}^4_4}^{1/4}.
\end{flalign*}In particular, we have $(\p_z v)(x,y,z,t_0) \in L^4(\O)$.

Recall that $(v(x,y,z,t_0),\r(x,y,z,t_0)) \in L^\infty(\O)$, it follows from Proposition \ref{po:J1t} that
\begin{flalign}
  &\sup_{t_0 \leq s \leq t}\xkh{\norm{(v,\r)}^2_{H^1(\O)}}(s)+\int^t_{t_0}{\norm{\nh v}^2_{H^1(\O)}}ds\nonumber\\
  &\qquad+\int^t_{t_0}{\xkh{{\norm{\nh \r}^2_{H^1(\O)}}+\norm{(\p_t v,\p_t \r)}^2_{L^2(\O)}}}ds \leq \ma{J}_1(t),\label{fl:t01}
\end{flalign}for any $t \in [t_0,\infty)$. Combining (\ref{fl:0t0}) and (\ref{fl:t01}) yields the energy estimate in Theorem \ref{th:J2t}. Hence the proof is completed.
\end{proof}

The proof Theorem~\ref{th:J3t} is shown below.
\begin{proof}[Proof of Theorem~\ref{th:J3t}.]
Thanks to the proof of Proposition \ref{po:eta}, substituting (\ref{fl:e22}) and the energy estimate in Theorem~\ref{th:J2t} into (\ref{fl:Ft}) gives
\begin{flalign*}
  &\sup_{0 \leq t \leq \ma{T}}\xkh{\norm{(V_\e,\e W_\e,\g_\e)}^2_2}(t)+\int^\ma{T}_0{\xkh{\norm{\nh V_\e}^2_2+\e^{\be-2}\norm{\p_z V_\e}^2_2}}dt\\
  &\qquad\quad+\int^\ma{T}_0{\xkh{\e^2\norm{\nh W_\e}^2_2+\norm{\nh \g_\e}^2_2+\e^\be\norm{\p_z W_\e}^2_2+\e^{\ga-2}\norm{\p_z \g_\e}^2_2}}dt\\
  &\qquad\leq C\exp\dkh{C(\ma{T}+1)\xkh{\ma{J}_2(\ma{T})+\ma{J}^2_2(\ma{T})}}\times\bigg\{\xkh{\e^2+\e^\be}\ma{J}_2(\ma{T})+\e^2\ma{J}^2_2(\ma{T})\\
  &\qquad\quad+\xkh{\e^{\ga-2}+\e^{\be-2}}\ma{T}\ma{J}_2(\ma{T})+\e^2(\ma{T}+1)\xkh{\norm{v_0}^2_2+\e^2\norm{w_0}^2_2+\norm{\r_0}^2_2}^2\bigg\}\\
  &\qquad\leq C\e^\eta(\ma{T}+1)\zkh{\ma{J}_2(\ma{T})+\ma{J}^2_2(\ma{T})+\xkh{\norm{v_0}^2_2+\norm{w_0}^2_2+\norm{\r_0}^2_2}^2}\\
  &\qquad\quad\times\exp\dkh{C(\ma{T}+1)\xkh{\ma{J}_2(\ma{T})+\ma{J}^2_2(\ma{T})}}=:\e^\eta\ma{J}_3(\ma{T}),
\end{flalign*}for any $\ma{T}>0$, where $\eta=\min\{2,\be-2,\ga-2\}$ with $2<\be,\ga<\infty$, and $\ma{J}_3(t)$ is a nonnegative continuously increasing function defined on $[0,\infty)$. Here $C$ is a positive constant that does not depend on $\e$.~Obviously, the local strong convergences in Theorem \ref{th:VTC} can be extended to the global strong convergences, and the rate of convergence is of the order $O(\e^{\eta/2})$.
\end{proof}

\section{Strong convergence for $H^2$ initial data}
In this section,~assume that the initial data $(v_0,\r_0) \in H^2(\O)$ with
\begin{equation*}
  \dz{\nh \d v_0(x,y,z)}=0,~\textnormal{for all}~(x,y)\in M,
\end{equation*}we prove that the scaled Boussinesq equations with rotation (\ref{eq:ptv}) strongly converge to the primitive equations with only horizontal viscosity and diffusivity (\ref{eq:ptl}) as the aspect ration parameter $\e$ goes to zero.

With this assumption of the initial data, there is a unique local strong solution $(v_\e,w_\e,\r_\e)$ to the system (\ref{eq:ptv}), subject to (\ref{ga:are})-(\ref{eq:eve}).~Denote by $\ma{T}^*_\e$ the maximal existence time of this local strong solution.~Moreover, the system (\ref{eq:ptl}) corresponding to (\ref{ga:are})-(\ref{eq:eve}) has a unique global strong solution $(v,\r)$.~The global well-posedness of strong solutions to the primitive equations with only horizontal viscosity and diffusivity (\ref{eq:ptl}) is as follows (see \cite{es2016}).
\begin{proposition}\label{po:J4t}
Let $v_0$,$T_0\in H^2(\O)$ be two periodic functions with~$\dz{\nh \d v_0(x,y,z)}=0,~\textnormal{for all}~(x,y)\in M$. Then the following assertions hold true:\\
(i) For any $\mathcal{T}>0$, there exists a unique global strong solution $(v,\r)$ to the primitive equations with only horizontal viscosity and diffusivity (\ref{eq:ptl}) corresponding to (\ref{ga:are})-(\ref{eq:eve}), such that
\begin{gather*}
  (v,\r) \in L^{\infty}\xkh{[0,\mathcal{T}];H^2(\O)} \cap C\xkh{[0,\mathcal{T}];H^1(\O)},\\
  (\nh v,\nh\r) \in L^2\xkh{[0,\mathcal{T}];H^2(\O)},(\p_t v,\p_t \r) \in L^2\xkh{[0,\mathcal{T}];H^1(\O)};
\end{gather*}
(ii) The global strong solution $(v,\r)$ satisfies the following energy estimate
\begin{flalign}
  &\sup_{0 \leq s \leq t}\xkh{\norm{(v,\r)}^2_{H^2(\O)}}(s)+\ds{\norm{\nh v}^2_{H^2(\O)}}\nonumber\\
  &\qquad+\ds{\xkh{{\norm{\nh \r}^2_{H^2(\O)}}+\norm{(\p_t v,\p_t \r)}^2_{H^1(\O)}}} \leq \ma{J}_4(t),\label{fl:J4t}
\end{flalign}for any $t \in [0,\infty)$, where $\ma{J}_4(t)$ is a nonnegative continuously increasing function defined on $[0,\infty)$.
\end{proposition}

Let
\begin{gather*}
  (U_\e,\g_\e,P_\e)=(V_\e,W_\e,\g_\e,P_\e), \\
  (V_\e,W_\e,\g_\e,P_\e)=(v_\e-v,w_\e-w,\r_\e-\r,p_\e-p).
\end{gather*}We subtract the system (\ref{eq:ptl}) from the system (\ref{eq:ptv}) and then lead to the following system
\begin{flalign}
  &\p_t V_\e-\la_h V_\e-\e^{\be-2}\p_{zz}V_\e-\e^{\be-2}\p_{zz}v+(u \d \n)V_\e+(U_\e \d \n)v\nonumber\\
  &+(U_\e \d \n)V_\e+\nh P_\e+f_0\vec{k} \times V_\e=0, \label{fl:Ve0}\\
  &\e^2(\p_t W_\e-\la_h W_\e-\e^{\be-2}\p_{zz}W_\e+u \d \n w+u \d \n W_\e+U_\e \d \n w)\nonumber\\
  &+\e^2U_\e \d \n W_\e+\p_z P_\e-\g_\e+\e^2(\p_t w-\la_h w+\e^{\be-2}\nh \d \p_z v)=0, \label{fl:dnw}\\
  &\p_t \g_\e-\la_h \g_\e-\e^{\ga-2}\p_{zz}\g_\e-\e^{\ga-2}\p_{zz}\r+u \d \n \g_\e\nonumber\\
  &+U_\e \d \n \r+U_\e \d \n \g_\e+W_\e=0, \label{fl:nge}\\
  &\nh \d V_\e+\p_z W_\e=0, \label{fl:zWe}
\end{flalign}defined on $\O \times (0,\ma{T}^*_\e)$.

\begin{proposition}\label{po:J5t}
Suppose that $(v_0,\r_0) \in H^2(\O)$, with $\dz{\nh \d v_0}=0$. Then the system (\ref{fl:Ve0})-(\ref{fl:zWe}) has the following basic energy estimate
\begin{flalign*}
  &\sup_{0 \leq s \leq t}\xkh{\norm{(V_\e,\e W_\e,\g_\e)}^2_2}(s)+\ds{\xkh{\norm{\nh V_\e}^2_2+\e^{\be-2}\norm{\p_z V_\e}^2_2}}\\
  &\qquad\quad+\ds{\xkh{\e^2\norm{\nh W_\e}^2_2+\norm{\nh \g_\e}^2_2+\e^\be\norm{\p_z W_\e}^2_2+\e^{\ga-2}\norm{\p_z \g_\e}^2_2}}\leq \e^\eta\ma{J}_5(t),
\end{flalign*}for any $t \in [0,\ma{T}^*_\e)$, where
\begin{equation*}
  \ma{J}_5(t)=C(t+1)e^{C(t+1)\xkh{\ma{J}_4(t)+\ma{J}^2_4(t)}}\zkh{\ma{J}_4(t)+\ma{J}^2_4(t)+\xkh{\norm{v_0}^2_2+\norm{w_0}^2_2+\norm{\r_0}^2_2}^2}.
\end{equation*}Here $\eta=\min\{2,\be-2,\ga-2\}$ with $2<\be,\ga<\infty$ and $C$ is a positive constant that does not depend on $\e$.
\end{proposition}

The proof of Proposition \ref{po:J5t} is similar to that of Theorem~\ref{th:J3t} and so is omitted. Note that the energy estimate (\ref{fl:J4t}) is used in this case.~With the aid of Proposition \ref{po:J4t} and \ref{po:J5t}, we can perform the first order energy estimate on the system (\ref{fl:Ve0})-(\ref{fl:zWe}) under some smallness condition.

\begin{proposition}\label{po:J6t}
Suppose that $(v_0,\r_0) \in H^2(\O)$, with $\dz{\nh \d v_0}=0$. Then there exists a small positive constant $\ell_0$ such that the system (\ref{fl:Ve0})-(\ref{fl:zWe}) has the following first order energy estimate
\begin{flalign*}
  &\sup_{0 \leq s \leq t}\xkh{\norm{\n(V_\e,\e W_\e,\g_\e)}^2_2}(s)+\ds{\xkh{\norm{\n\nh V_\e}^2_2+\e^{\be-2}\norm{\n\p_z V_\e}^2_2+\e^2\norm{\n\nh W_\e}^2_2}}\\
  &\qquad\quad+\ds{\xkh{\norm{\n\nh \g_\e}^2_2+\e^\be\norm{\n\p_z W_\e}^2_2+\e^{\ga-2}\norm{\n\p_z \g_\e}^2_2}}\leq \e^\eta\ma{J}_6(t),
\end{flalign*}for any $t \in [0,\ma{T}^*_\e)$, provided that
\begin{equation*}
  \sup_{0 \leq s \leq t}\xkh{\norm{\n (V_\e,\g_\e)}^2_2+\e^2\norm{\n W_\e}^2_2}(s) \leq \ell^2_0,
\end{equation*}where
\begin{flalign*}
  \ma{J}_6(t)&\leq C(t+1)\zkh{\ma{J}_4(t)+\ma{J}^2_4(t)+\ma{J}_5(t)+\ma{J}^2_5(t)+\ma{J}_4(t)\ma{J}_5(t)}\\
  &\quad\times\exp\dkh{C(t+1)\zkh{\ma{J}_4(t)+\ma{J}^2_4(t)+\ma{J}_5(t)+\ma{J}^2_5(t)+1}}.
\end{flalign*}Here $\eta=\min\{2,\be-2,\ga-2\}$ with $2<\be,\ga<\infty$ and $C$ is a positive constant that does not depend on $\e$.
\end{proposition}
\begin{proof}[Proof.]
Taking the $L^2(\O)$ inner product of the first three equation in system~(\ref{fl:Ve0})-(\ref{fl:zWe}) with $-\la V_\e$, $-\la W_\e$ and $-\la \g_\e$, respectively, then it follows from integration by parts that
\begin{flalign}
  \frac{1}{2}\frac{d}{dt}&\xkh{\norm{\n (V_\e,\e W_\e,\g_\e)}^2_2}
  +\xkh{\norm{\n\nh V_\e}^2_2+\e^{\be-2}\norm{\n\p_z V_\e}^2_2+\e^2\norm{\n\nh W_\e}^2_2}\nonumber\\
  &\quad+\xkh{\norm{\n\nh \g_\e}^2_2+\e^\be\norm{\n\p_z W_\e}^2_2+\e^{\ga-2}\norm{\n\p_z \g_\e}^2_2}\nonumber\\
  &=\oo{(u \d \n \g_\e+U_\e \d \n \r+U_\e \d \n \g_\e)\la \g_\e}\nonumber\\
  &\quad+\e^2\oo{\xkh{u \d \n w+u \d \n W_\e+U_\e \d \n w+U_\e \d \n W_\e}\la W_\e}\nonumber\\
  &\quad+\oo{\zkh{(u \d \n)V_\e+(U_\e \d \n)v+(U_\e \d \n)V_\e} \d \la V_\e}\nonumber\\
  &\quad+\oo{\zkh{\e^2\xkh{\p_t w-\la_h w}\la W_\e-\e^{\ga-2}(\p_{zz}\r)\la\g_\e}}\nonumber\\
  &\quad+\oo{\zkh{\e^\be(\nh \d \p_z v)\la W_\e-\e^{\be-2}(\p_{zz}v)\d\la V_\e}}\nonumber\\
  &=:\ma{I}_1+\ma{I}_2+\ma{I}_3+\ma{I}_4+\ma{I}_5,\label{fl:Veg}
\end{flalign}note that the resultants have been added up.

In order to estimate the first integral term $\ma{I}_1$ on the right-hand side of (\ref{fl:Veg}), we split it into three parts, $\ma{I}_{11}$, $\ma{I}_{12}$ and $\ma{I}_{13}$. By virtue of the divergence-free condition, Lemma \ref{le:phi} and Young inequality, these integral terms can be bounded as
\begin{flalign*}
  \ma{I}_{11}:&=\oo{(u \d \n \g_\e)\la \g_\e}\\
  &=\oo{[(v \d \nh \g_\e)\la_h \g_\e+w(\p_z \g_\e)\la_h \g_\e]}\\
  &\quad+\oo{[(\nh \d v)(\p_z \g_\e)\p_z \g_\e-(\p_z v \d \nh \g_\e)\p_z \g_\e]}\\
  &\leq\mm{\xkh{\dz{(|v|+|\p_z v|)}}\xkh{\dz{|\nh \g_\e||\la_h \g_\e|}}}\\
  &\quad+\mm{\xkh{\dz{|\nh v|}}\xkh{\dz{|\p_z \g_\e||\la_h \g_\e|}}}\\
  &\quad+\mm{\xkh{\dz{(|\nh v|+|\nh \p_z v|)}}\xkh{\dz{|\p_z \g_\e|^2}}}\\
  &\quad+\mm{\xkh{\dz{(|\nh \g_\e|+|\nh \p_z \g_\e|)}}\xkh{\dz{|\p_z v||\p_z \g_\e|}}}\\
  &\leq C\zkh{\norm{v}^{1/2}_2\xkh{\norm{v}^{1/2}_2+\norm{\nh v}^{1/2}_2}
  +\norm{\p_z v}^{1/2}_2\xkh{\norm{\p_z v}^{1/2}_2+\norm{\nh \p_z v}^{1/2}_2}}\\
  &\quad\times\norm{\nh \g_\e}^{1/2}_2\xkh{\norm{\nh \g_\e}^{1/2}_2+\norm{\nh^2 \g_\e}^{1/2}_2}\norm{\la_h \g_\e}_2\\
  &\quad+C\norm{\p_z \g_\e}^{1/2}_2\norm{\nh \p_z \g_\e}^{1/2}_2\norm{\la_h \g_\e}_2
  \norm{\nh v}^{1/2}_2\xkh{\norm{\nh v}^{1/2}_2+\norm{\nh^2 v}^{1/2}_2}\\
  &\quad+C\norm{\nh v}^{1/2}_2\xkh{\norm{\nh v}^{1/2}_2+\norm{\nh^2 v}^{1/2}_2}\norm{\p_z \g_\e}_2\norm{\la_h \g_\e}_2\\
  &\quad+C(\norm{\nh v}_2+\norm{\nh \p_z v}_2)\norm{\p_z \g_\e}_2(\norm{\p_z \g_\e}_2+\norm{\nh \p_z \g_\e})\\
  &\quad+C(\norm{\nh \g_\e}_2+\norm{\nh \p_z \g_\e}_2)\norm{\p_z v}^{1/2}_2\xkh{\norm{\p_z v}^{1/2}_2+\norm{\nh \p_z v}^{1/2}_2}\\
  &\quad\times\norm{\p_z \g_\e}^{1/2}_2\xkh{\norm{\p_z \g_\e}^{1/2}_2+\norm{\nh \p_z \g_\e}^{1/2}_2}\\
  &\leq C\xkh{\norm{v}^2_2+\norm{\nh v}^2_2+\norm{\n v}^2_2+\norm{\n\nh v}^2_2+\norm{v}^4_2}\norm{\n \g_\e}^2_2\\
  &\quad+C\xkh{\norm{v}^2_2\norm{\nh v}^2_2+\norm{\n v}^4_2+\norm{\n v}^2_2\norm{\nh v}^2_2}\norm{\n \g_\e}^2_2\\
  &\quad+C\xkh{1+\norm{\n v}^2_2\norm{\n\nh v}^2_2}\norm{\n \g_\e}^2_2+\frac{2}{105}\norm{\n\nh \g_\e}^2_2,
\end{flalign*}
\begin{flalign*}
  \ma{I}_{12}:&=\oo{(U_\e \d \n \r)\la \g_\e}\\
  &=\oo{[(V_\e \d \nh \r)\la_h \g_\e-(\p_z V_\e \d \nh \r)\p_z \g_\e-(V_\e \d \nh \p_z \r)\p_z \g_\e]}\\
  &\quad+\oo{[W_\e(\p_z \r)\la_h \g_\e-(\p_z W_\e)(\p_z \r)\p_z \g_\e-W_\e(\p_{zz} \r)\p_z \g_\e]}\\
  &\leq\mm{\xkh{\dz{(|V_\e|+|\p_z V_\e|)}}\xkh{\dz{|\nh \r||\la_h \g_\e|}}}\\
  &\quad+\mm{\xkh{\dz{(|\nh \r|+|\nh \p_z \r|)}}\xkh{\dz{|\p_z V_\e||\p_z \g_\e|}}}\\
  &\quad+\mm{\xkh{\dz{(|V_\e|+|\p_z V_\e|)}}\xkh{\dz{|\p_z \g_\e||\nh \p_z \r|}}}\\
  &\quad+\mm{\xkh{\dz{|\nh V_\e|}}\xkh{\dz{|\p_z \r||\la_h \g_\e|}}}\\
  &\quad+\mm{\xkh{\dz{(|\nh V_\e|+|\nh \p_z V_\e|)}}\xkh{\dz{|\p_z \r||\p_z \g_\e|}}}\\
  &\quad+\mm{\xkh{\dz{|\nh V_\e|}}\xkh{\dz{|\p_{zz} \r||\p_z \g_\e|}}}\\
  &\leq C\zkh{\norm{V_\e}^{1/2}_2\xkh{\norm{V_\e}^{1/2}_2+\norm{\nh V_\e}^{1/2}_2}
  +\norm{\p_z V_\e}^{1/2}_2\xkh{\norm{\p_z V_\e}^{1/2}_2+\norm{\nh \p_z V_\e}^{1/2}_2}}\\
  &\quad\times\norm{\nh \r}^{1/2}_2\xkh{\norm{\nh \r}^{1/2}_2+\norm{\nh^2 \r}^{1/2}_2}\norm{\la_h \g_\e}_2\\
  &\quad+C(\norm{\nh \r}_2+\norm{\nh \p_z \r}_2)\norm{\p_z V_\e}^{1/2}_2\xkh{\norm{\p_z V_\e}^{1/2}_2+\norm{\nh \p_z V_\e}^{1/2}_2}\\
  &\quad\times\norm{\p_z \g_\e}^{1/2}_2\xkh{\norm{\p_z \g_\e}^{1/2}_2+\norm{\nh \p_z \g_\e}^{1/2}_2}\\
  &\quad+C\zkh{\norm{V_\e}^{1/2}_2\xkh{\norm{V_\e}^{1/2}_2+\norm{\nh V_\e}^{1/2}_2}
  +\norm{\p_z V_\e}^{1/2}_2\xkh{\norm{\p_z V_\e}^{1/2}_2+\norm{\nh \p_z V_\e}^{1/2}_2}}\\
  &\quad\times\norm{\p_z \g_\e}^{1/2}_2\xkh{\norm{\p_z \g_\e}^{1/2}_2+\norm{\nh \p_z \g_\e}^{1/2}_2}\norm{\nh \p_z \r}_2\\
  &\quad+C\norm{\nh V_\e}^{1/2}_2\xkh{\norm{\nh V_\e}^{1/2}_2+\norm{\nh^2 V_\e}^{1/2}_2}
  \norm{\p_z \r}^{1/2}_2\norm{\nh \p_z \r}^{1/2}_2\norm{\la_h \g_\e}_2\\
  &\quad+C\norm{\p_z \r}_2\norm{\la_h \g_\e}_2\norm{\nh V_\e}^{1/2}_2\xkh{\norm{\nh V_\e}^{1/2}_2+\norm{\nh^2 V_\e}^{1/2}_2}\\
  &\quad+C(\norm{\nh V_\e}_2+\norm{\nh \p_z V_\e}_2)\norm{\p_z \r}^{1/2}_2\xkh{\norm{\p_z \r}^{1/2}_2+\norm{\nh \p_z \r}^{1/2}_2}\\
  &\quad\times\norm{\p_z \g_\e}^{1/2}_2\xkh{\norm{\p_z \g_\e}^{1/2}_2+\norm{\nh \p_z \g_\e}^{1/2}_2}\\
  &\quad+C\norm{\nh V_\e}^{1/2}_2\xkh{\norm{\nh V_\e}^{1/2}_2+\norm{\nh^2 V_\e}^{1/2}_2}
  \norm{\p_z \g_\e}^{1/2}_2\norm{\nh \p_z \g_\e}^{1/2}_2\norm{\p_{zz} \r}_2\\
  &\quad+C\norm{\p_z \g_\e}_2\norm{\p_{zz} \r}_2\norm{\nh V_\e}^{1/2}_2\xkh{\norm{\nh V_\e}^{1/2}_2+\norm{\nh^2 V_\e}^{1/2}_2}\\
  &\leq C\xkh{\norm{\nh \r}^2_2+\norm{\n \r}^2_2+\norm{\n \r}^4_2+\norm{\n^2 \r}^2_2}\xkh{\norm{\n V_\e}^2_2+\norm{\n \g_\e}^2_2}\\
  &\quad+C\xkh{1+\norm{\n\nh \r}^2_2+\norm{\n \r}^2_2\norm{\n \nh \r}^2_2}\xkh{\norm{\n V_\e}^2_2+\norm{\n \g_\e}^2_2}\\
  &\quad+C\norm{V_\e}^2_2\xkh{1+\norm{\n \nh \r}^2_2}+\frac{2}{105}\xkh{\norm{\n\nh V_\e}^2_2+\norm{\n\nh \g_\e}^2_2},
\end{flalign*}and
\begin{flalign*}
  \ma{I}_{13}:&=\oo{(U_\e \d \n \g_\e)\la \g_\e}\\
  &=\oo{[(V_\e \d \nh \g_\e)\la_h \g_\e+W_\e(\p_z \g_\e)\la_h \g_\e]}\\
  &\quad+\oo{[(V_\e \d \nh \g_\e)\p_{zz} \g_\e+W_\e(\p_z \g_\e)\p_{zz} \g_\e]}\\
  &=\oo{[(V_\e \d \nh \g_\e)\la_h \g_\e+W_\e(\p_z \g_\e)\la_h \g_\e]}\\
  &\quad+\oo{[(\nh \d V_\e)(\p_z \g_\e)\p_z \g_\e-(\p_z V_\e \d \nh \g_\e)\p_z \g_\e]}\\
  &\leq\mm{\xkh{\dz{(|V_\e|+|\p_z V_\e|)}}\xkh{\dz{|\nh \g_\e||\la_h \g_\e|}}}\\
  &\quad+\mm{\xkh{\dz{|\nh V_\e|}}\xkh{\dz{|\p_z \g_\e||\la_h \g_\e|}}}\\
  &\quad+\mm{\xkh{\dz{(|\nh V_\e|+|\nh \p_z V_\e|)}}\xkh{\dz{|\p_z \g_\e|^2}}}\\
  &\quad+\mm{\xkh{\dz{(|\nh \g_\e|+|\nh \p_z \g_\e|)}}\xkh{\dz{|\p_z V_\e||\p_z \g_\e|}}}\\
  &\leq C\zkh{\norm{V_\e}^{1/2}_2\xkh{\norm{V_\e}^{1/2}_2+\norm{\nh V_\e}^{1/2}_2}
  +\norm{\p_z V_\e}^{1/2}_2\xkh{\norm{\p_z V_\e}^{1/2}_2+\norm{\nh \p_z V_\e}^{1/2}_2}}\\
  &\quad\times\norm{\nh \g_\e}^{1/2}_2\xkh{\norm{\nh \g_\e}^{1/2}_2+\norm{\nh^2 \g_\e}^{1/2}_2}\norm{\la_h \g_\e}_2\\
  &\quad+C\norm{\nh V_\e}^{1/2}_2\xkh{\norm{\nh V_\e}^{1/2}_2+\norm{\nh^2 V_\e}^{1/2}_2}
  \norm{\p_z \g_\e}^{1/2}_2\norm{\nh \p_z \g_\e}^{1/2}_2\norm{\la_h \g_\e}_2\\
  &\quad+C\norm{\p_z \g_\e}_2\norm{\la_h \g_\e}_2\norm{\nh V_\e}^{1/2}_2\xkh{\norm{\nh V_\e}^{1/2}_2+\norm{\nh^2 V_\e}^{1/2}_2}\\
  &\quad+C(\norm{\nh V_\e}_2+\norm{\nh \p_z V_\e}_2)\norm{\p_z \g_\e}_2(\norm{\p_z \g_\e}_2+\norm{\nh \p_z \g_\e})\\
  &\quad+C(\norm{\nh \g_\e}_2+\norm{\nh \p_z \g_\e}_2)\norm{\p_z \g_\e}^{1/2}_2\xkh{\norm{\p_z \g_\e}^{1/2}_2+\norm{\nh \p_z \g_\e}^{1/2}_2}\\
  &\quad\times\norm{\p_z V_\e}^{1/2}_2\xkh{\norm{\p_z V_\e}^{1/2}_2+\norm{\nh \p_z V_\e}^{1/2}_2}\\
  &\leq C\zkh{1+\norm{V_\e}^4_2+\norm{V_\e}^2_2\norm{\nh V_\e}^2_2+\xkh{\norm{\n\nh V_\e}^2_2+\norm{\n\nh \g_\e}^2_2}}\\
  &\quad\times\xkh{\norm{\n V_\e}^2_2+\norm{\n \g_\e}^2_2}+\frac{2}{105}\xkh{\norm{\n\nh V_\e}^2_2+\norm{\n\nh \g_\e}^2_2},
\end{flalign*}respectively, where we have used the boundary condition (\ref{ga:are}) and symmetry condition (\ref{eq:eve}). Combining the estimates for $\ma{I}_{11}$, $\ma{I}_{12}$ and $\ma{I}_{13}$ gives
\begin{flalign}
  \ma{I}_1:&=\oo{(u \d \n \g_\e+U_\e \d \n \r+U_\e \d \n\g_\e)\la \g_\e}\nonumber\\
  &\leq C\bigg\{\xkh{\norm{v}^2_2+\norm{\nh v}^2_2+\norm{\n v}^2_2+\norm{\n\nh v}^2_2+\norm{v}^4_2+\norm{v}^2_2\norm{\nh v}^2_2}\nonumber\\
  &\quad+\xkh{\norm{\n v}^4_2+\norm{\n v}^2_2\norm{\n\nh v}^2_2+\norm{\n v}^2_2\norm{\nh v}^2_2}\nonumber\\
  &\quad+\xkh{\norm{\nh \r}^2_2+\norm{\n \r}^2_2+\norm{\n \r}^4_2+\norm{\n^2 \r}^2_2+\norm{\n\nh \r}^2_2}\nonumber\\
  &\quad+\xkh{\norm{\n \r}^2_2\norm{\n \nh \r}^2_2+1+\norm{V_\e}^4_2+\norm{V_\e}^2_2\norm{\nh V_\e}^2_2}\nonumber\\
  &\quad+\xkh{\norm{\n\nh V_\e}^2_2+\norm{\n\nh \g_\e}^2_2}\bigg\}\times\xkh{\norm{\n V_\e}^2_2+\norm{\n \g_\e}^2_2}\nonumber\\
  &\quad+C\norm{V_\e}^2_2\xkh{1+\norm{\n \nh \r}^2_2}+\frac{2}{35}\xkh{\norm{\n\nh V_\e}^2_2+\norm{\n\nh \g_\e}^2_2}.\label{fl:I1}
\end{flalign}Using the similar method as the first integral term $\ma{I}_1$ on the right-hand side of (\ref{fl:Veg}), the integral terms $\ma{I}_2$ and $\ma{I}_3$ on the right-hand side of (\ref{fl:Veg}) can be estimated as
\begin{flalign}
  \ma{I}_2:&=\e^2\oo{\xkh{u\d\n w+u \d \n W_\e+U_\e \d \n w+U_\e \d \n W_\e}\la W_\e}\nonumber\\
  &=\e^2\oo{\zkh{(\nh \d v)\dk{(\nh \d v)}-v \d \dk{\nh(\nh \d v)}}\la_h W_\e}\nonumber\\
  &\quad+\e^2\oo{\zkh{(\nh \d v)(\nh \d v)-\p_z v \d \dk{\nh(\nh \d v)}}(\nh \d V_\e)}\nonumber\\
  &\quad+\e^2\oo{\zkh{(\nh \d \p_z v)\dk{(\nh \d v)}-v \d \nh(\nh \d v)}(\nh \d V_\e)}\nonumber\\
  &\quad+\e^2\oo{\zkh{v \d \nh W_\e+(\nh \d V_\e)\dk{(\nh \d v)}}\la_h W_\e}\nonumber\\
  &\quad+\e^2\oo{\zkh{\p_z v \d \dk{\nh(\nh \d V_\e)}-2v \d \nh \p_z W_\e}\p_z W_\e}\nonumber\\
  &\quad+\e^2\oo{\zkh{(\nh \d v)\dk{(\nh \d V_\e)}-V_\e \d \dk{\nh(\nh \d v)}}\la_h W_\e}\nonumber\\
  &\quad+\e^2\oo{\zkh{\p_z V_\e \d \dk{\nh(\nh \d v)}+(\nh \d v)\p_z W_\e}\p_z W_\e}\nonumber\\
  &\quad+\e^2\oo{\zkh{V_\e \d \nh(\nh \d v)+(\nh \d \p_z v)W_\e}\p_z W_\e}\nonumber\\
  &\quad+\e^2\oo{\zkh{V_\e \d \nh W_\e-\p_z W_\e\dk{(\nh \d V_\e)}}\la_h W_\e}\nonumber\\
  &\quad+\e^2\oo{\zkh{\p_z V_\e \d \dk{\nh(\nh \d V_\e)}-2V_\e \d \nh \p_z W_\e}\p_z W_\e}\nonumber\\
  &\leq C\bigg\{\zkh{1+\e^2+\norm{v}^2_2+\norm{\nh v}^2_2+(1+\e^2)\norm{\n v}^2_2+(1+\e^2)\norm{\n\nh v}^2_2}\nonumber\\
  &\quad+\zkh{(1+\e^4)\norm{\n v}^4_2+\e^2\norm{\n^2\nh v}^2_2+\e^4\norm{\n^2 v}^2_2\norm{\n^2\nh v}^2_2}\nonumber\\
  &\quad+\zkh{\norm{v}^4_2+\norm{v}^2_2\norm{\nh v}^2_2+(1+\e^4)\norm{\n v}^2_2\norm{\n\nh v}^2_2}\nonumber\\
  &\quad+\xkh{\norm{V_\e}^2_2+\norm{\nh V_\e}^2_2+\e^2\norm{\nh W_\e}^2_2+\norm{V_\e}^4_2+\norm{V_\e}^2_2\norm{\nh V_\e}^2_2}\nonumber\\
  &\quad+\xkh{\norm{\n\nh V_\e}^2_2+\e^2\norm{\n\nh W_\e}^2_2}\bigg\}\times\xkh{\norm{\n V_\e}^2_2+\e^2\norm{\n W_\e}^2_2}\nonumber\\
  &\quad+C\e^2\xkh{\norm{v}^2_2+\norm{\nh v}^2_2+\norm{\n v}^2_2+\norm{\n\nh v}^2_2+\norm{\n^2\nh v}^2_2}\nonumber\\
  &\quad+C\e^2\xkh{\norm{v}^4_2+\norm{v}^2_2\norm{\nh v}^2_2+\norm{\n^2 v}^2_2\norm{\n^2\nh v}^2_2}\nonumber\\
  &\quad+C\xkh{\norm{V_\e}^4_2+\norm{V_\e}^2_2\norm{\nh V_\e}^2_2}+\frac{2}{35}\xkh{\norm{\n\nh V_\e}^2_2+\e^2\norm{\n\nh W_\e}^2_2},\label{fl:I2}
\end{flalign}and
\begin{flalign}
  \ma{I}_3:&=\oo{\zkh{(u \d \n)V_\e+(U_\e \d \n)v+(U_\e \d \n)V_\e} \d \la V_\e}\nonumber\\
  &=\oo{\zkh{(v \d \nh)V_\e-(\p_z V_\e)\dk{(\nh \d v)}} \d \la_h V_\e}\nonumber\\
  &\quad+\oo{\zkh{(-\p_z v \d \nh)V_\e-2(v \d \nh)\p_z V_\e} \d \p_z V_\e}\nonumber\\
  &\quad+\oo{\zkh{(V_\e \d \nh)v-(\p_z v)\dk{(\nh \d V_\e)}} \d \la_h V_\e}\nonumber\\
  &\quad+\oo{\zkh{(\nh \d V_\e)\p_z v-(\p_z V_\e \d \nh)v} \d \p_z V_\e}\nonumber\\
  &\quad+\oo{\zkh{(\p_{zz} v)\dk{(\nh \d V_\e)}-(V_\e \d \nh)\p_z v} \d \p_z V_\e}\nonumber\\
  &\quad+\oo{\zkh{(V_\e \d \nh)V_\e-(\p_z V_\e)\dk{(\nh \d V_\e)}} \d \la_h V_\e}\nonumber\\
  &\quad+\oo{\zkh{(\nh \d V_\e)\p_z V_\e-(\p_z V_\e \d \nh)V_\e} \d \p_z V_\e}\nonumber\\
  &\leq C\xkh{1+\norm{v}^2_2+\norm{\nh v}^2_2+\norm{\n v}^2_2+\norm{\n\nh v}^2_2+\norm{\n^2 v}^2_2}\norm{\n V_\e}^2_2\nonumber\\
  &\quad+C\xkh{\norm{v}^4_2+\norm{v}^2_2\norm{\nh v}^2_2+\norm{\n v}^4_2+\norm{\n v}^2_2\norm{\n\nh v}^2_2}\norm{\n V_\e}^2_2\nonumber\\
  &\quad+\xkh{\norm{\nh V_\e}^2_2+\norm{V_\e}^4_2+\norm{V_\e}^2_2\norm{\nh V_\e}^2_2+\norm{\n\nh V_\e}^2_2}\norm{\n V_\e}^2_2\nonumber\\
  &\quad+C\norm{V_\e}^2_2\xkh{1+\norm{\n\nh v}^2_2}+\frac{2}{35}\norm{\n\nh V_\e}^2_2,\label{fl:I3}
\end{flalign}respectively. For the integral term $\ma{I}_4$ on the right-hand side of (\ref{fl:Veg}), we apply the H\"{o}lder inequality and Young inequality to obtain
\begin{flalign}
  \ma{I}_4:&=\oo{\zkh{\e^2\xkh{\p_t w-\la_h w}\la W_\e-\e^{\ga-2}(\p_{zz}\r)\la\g_\e}}\nonumber\\
  &=\oo{\zkh{\e^2\xkh{\p_t w-\la_h w}\la_h W_\e-\e^{\ga-2}(\p_{zz}\r)\la_h\g_\e}}\nonumber\\
  &\quad+\oo{\zkh{\e^{\ga-2}(-\p_{zz}\r)\p_{zz}\g_\e+\e^2\xkh{\p_t w-\la_h w}\p_{zz} W_\e}}\nonumber\\
  &\leq \e^2\oo{\xkh{\dz{\xkh{|\nh^3 v|+|\nh\p_t v|}}}|\la_h W_\e|}\nonumber\\
  &\quad+\e^{\ga-2}\norm{\p_{zz}\r}_2\norm{\la_h\g_\e}_2+\e^{\ga-2}\norm{\p_{zz}\r}_2\norm{\p_{zz}\g_\e}_2\nonumber\\
  &\quad+\e^2\oo{\xkh{|\nh^3 v|+|\nh\p_t v|}|\p_z W_\e|}\nonumber\\
  &\leq C\e^2\oo{\xkh{\dz{\xkh{|\nh^3 v|+|\nh\p_t v|}^2}}}+C\e^2\norm{\p_z W_\e}^2_2\nonumber\\
  &\quad+C\xkh{\e^{\ga-2}+\e^{2\ga-4}}\norm{\p_{zz}\r}^2_2+C\e^2\xkh{\norm{\nh^3 v}^2_2+\norm{\nh\p_t v}^2_2}\nonumber\\
  &\quad+\frac{2}{35}\xkh{\e^2\norm{\la_h W_\e}^2_2+\norm{\la_h \g_\e}^2_2+\e^{\ga-2}\norm{\p_{zz}\g_\e}^2_2}\nonumber\\
  &\leq C\e^2\xkh{\norm{\n^2\nh v}^2_2+\norm{\n \p_t v}^2_2+\norm{\n W_\e}^2_2}+C\xkh{\e^{\ga-2}+\e^{2\ga-4}}\norm{\n^2 \r}^2_2\nonumber\\
  &\quad+\frac{2}{35}\xkh{\e^2\norm{\n\nh W_\e}^2_2+\norm{\n\nh \g_\e}^2_2+\e^{\ga-2}\norm{\n\p_z\g_\e}^2_2},\label{fl:I4}
\end{flalign}where the divergence-free condition is used. A similar argument as the integral term $\ma{I}_4$ on the right-hand side of (\ref{fl:Veg}) yields
\begin{flalign}
  \ma{I}_5:&=\oo{\zkh{\e^\be(\nh \d \p_z v)\la W_\e-\e^{\be-2}(\p_{zz}v)\d \la V_\e}}\nonumber\\
  &=\oo{\zkh{\e^\be(\nh \d \p_z v)\la_h W_\e-\e^{\be-2}(\p_{zz}v)\d\la_h V_\e}}\nonumber\\
  &\quad+\oo{\zkh{\e^\be(\nh \d \p_z v)\p_{zz} W_\e-\e^{\be-2}(\p_{zz}v)\d \p_{zz} V_\e}}\nonumber\\
  &\leq \e^\be\norm{\nh\p_z v}_2\norm{\la_h W_\e}_2+\e^{\be-2}\norm{\p_{zz}v}_2\norm{\la_h V_\e}_2\nonumber\\
  &\quad+\e^\be\norm{\nh\p_z v}_2\norm{\p_{zz} W_\e}_2+\e^{\be-2}\norm{\p_{zz}v}_2\norm{\p_{zz}V_\e}_2\nonumber\\
  &\leq C\xkh{\e^{\be-2}+\e^{2\be-4}}\norm{\n^2 v}^2_2+C\xkh{\e^\be+\e^{2\be-2}}\norm{\n\nh v}^2_2\nonumber\\
  &\quad+\frac{2}{35}\xkh{\norm{\n\nh V_\e}^2_2+\e^2\norm{\n\nh W_\e}^2_2+\e^\be\norm{\n\p_z W_\e}^2_2+\e^{\be-2}\norm{\n\p_z V_\e}^2_2}.\label{fl:I5}
\end{flalign}Substituting (\ref{fl:I1})-(\ref{fl:I5}) into (\ref{fl:Veg}) leads to
\begin{flalign*}
  \frac{1}{2}\frac{d}{dt}&\xkh{\norm{\n (V_\e,\e W_\e,\g_\e)}^2_2}
  +\frac{5}{7}\xkh{\norm{\n\nh V_\e}^2_2+\e^{\be-2}\norm{\n\p_z V_\e}^2_2+\e^2\norm{\n\nh W_\e}^2_2}\\
  &\quad+\frac{5}{7}\xkh{\norm{\n\nh \g_\e}^2_2+\e^\be\norm{\n\p_z W_\e}^2_2+\e^{\ga-2}\norm{\n\p_z \g_\e}^2_2}\\
  &\leq C_0\bigg\{\zkh{1+\e^2+\norm{v}^2_2+\norm{\nh v}^2_2+(1+\e^2)\norm{\n v}^2_2+(1+\e^2)\norm{\n\nh v}^2_2}\\
  &\quad+\zkh{\norm{\n^2 v}^2_2+(1+\e^4)\norm{\n v}^4_2+\e^2\norm{\n^2\nh v}^2_2+\e^4\norm{\n^2 v}^2_2\norm{\n^2\nh v}^2_2}\\
  &\quad+\zkh{\norm{v}^4_2+\norm{v}^2_2\norm{\nh v}^2_2+\norm{\n v}^2_2\norm{\nh v}^2_2+(1+\e^4)\norm{\n v}^2_2\norm{\n\nh v}^2_2}\\
  &\quad+\xkh{\norm{\nh \r}^2_2+\norm{\n \r}^2_2+\norm{\n \r}^4_2+\norm{\n^2 \r}^2_2+\norm{\n\nh \r}^2_2+\norm{\n \r}^2_2\norm{\n \nh \r}^2_2}\\
  &\quad+\xkh{\norm{V_\e}^2_2+\norm{\nh V_\e}^2_2+\e^2\norm{\nh W_\e}^2_2+\norm{V_\e}^4_2+\norm{V_\e}^2_2\norm{\nh V_\e}^2_2}\\
  &\quad+\xkh{\norm{\n\nh V_\e}^2_2+\e^2\norm{\n\nh W_\e}^2_2+\norm{\n\nh \g_\e}^2_2}\bigg\}\times\xkh{\norm{\n (V_\e,\e W_\e,\g_\e)}^2_2}\\
  &\quad+C_0\e^2\xkh{\norm{v}^2_2+\norm{\nh v}^2_2+\norm{\n v}^2_2+\norm{\n\nh v}^2_2+\norm{\n^2\nh v}^2_2}\\
  &\quad+C_0\e^2\xkh{\norm{v}^4_2+\norm{v}^2_2\norm{\nh v}^2_2+\norm{\n^2 v}^2_2\norm{\n^2\nh v}^2_2+\norm{\n \p_t v}^2_2}\\
  &\quad+C_0\xkh{\norm{V_\e}^4_2+\norm{V_\e}^2_2\norm{\nh V_\e}^2_2}+C_0\xkh{\e^{\ga-2}+\e^{2\ga-4}}\norm{\n^2 \r}^2_2\\
  &\quad+C_0\xkh{\e^{\be-2}+\e^{2\be-4}}\norm{\n^2 v}^2_2+C_0\xkh{\e^\be+\e^{2\be-2}}\norm{\n\nh v}^2_2\\
  &\quad+C_0\xkh{\norm{V_\e}^2_2+\norm{V_\e}^2_2\norm{\n\nh v}^2_2+\norm{V_\e}^2_2\norm{\n\nh \r}^2_2}.
\end{flalign*}Using the assumption given by the proposition
\begin{equation*}
  \sup_{0 \leq s \leq t}\xkh{\norm{\n (V_\e,\g_\e)}^2_2+\e^2\norm{\n W_\e}^2_2}(s) \leq \ell^2_0,
\end{equation*}and choosing $\ell_0=\sqrt{\frac{3}{14C_0}}$, it deduces from the above inequality that
\begin{flalign*}
  \frac{d}{dt}&\xkh{\norm{\n (V_\e,\e W_\e,\g_\e)}^2_2}+\xkh{\norm{\n\nh V_\e}^2_2+\e^{\be-2}\norm{\n\p_z V_\e}^2_2+\e^2\norm{\n\nh W_\e}^2_2}\\
  &\quad+\xkh{\norm{\n\nh \g_\e}^2_2+\e^\be\norm{\n\p_z W_\e}^2_2+\e^{\ga-2}\norm{\n\p_z \g_\e}^2_2}\\
  &\leq C_0\bigg\{\zkh{\norm{v}^2_2+\norm{\nh v}^2_2+(1+\e^2)\norm{\n v}^2_2+(1+\e^2)\norm{\n\nh v}^2_2}\\
  &\quad+\zkh{\e^2\norm{\n^2\nh v}^2_2+\e^4\norm{\n^2 v}^2_2\norm{\n^2\nh v}^2_2+(1+\e^4)\norm{\n v}^4_2}\\
  &\quad+\zkh{(1+\e^2)+\norm{v}^4_2+\norm{\n^2 v}^2_2+\norm{v}^2_2\norm{\nh v}^2_2+\norm{\n v}^2_2\norm{\nh v}^2_2}\\
  &\quad+\zkh{\norm{\nh \r}^2_2+\norm{\n \r}^2_2+\norm{\n \r}^4_2+(1+\e^4)\norm{\n v}^2_2\norm{\n\nh v}^2_2}\\
  &\quad+\xkh{\norm{\n^2 \r}^2_2+\norm{\n\nh \r}^2_2+\norm{\n \r}^2_2\norm{\n \nh \r}^2_2+\norm{V_\e}^2_2+\norm{\nh V_\e}^2_2}\\
  &\quad+\xkh{\e^2\norm{\nh W_\e}^2_2+\norm{V_\e}^4_2+\norm{V_\e}^2_2\norm{\nh V_\e}^2_2}\bigg\}\times\xkh{\norm{\n (V_\e,\e W_\e,\g_\e)}^2_2}\\
  &\quad+C_0\e^2\xkh{\norm{v}^2_2+\norm{\nh v}^2_2+\norm{\n v}^2_2+\norm{\n\nh v}^2_2+\norm{\n^2\nh v}^2_2}\\
  &\quad+C_0\e^2\xkh{\norm{v}^4_2+\norm{v}^2_2\norm{\nh v}^2_2+\norm{\n^2 v}^2_2\norm{\n^2\nh v}^2_2+\norm{\n \p_t v}^2_2}\\
  &\quad+C_0\xkh{\norm{V_\e}^4_2+\norm{V_\e}^2_2\norm{\nh V_\e}^2_2}+C_0\xkh{\e^{\ga-2}+\e^{2\ga-4}}\norm{\n^2 \r}^2_2\\
  &\quad+C_0\xkh{\e^{\be-2}+\e^{2\be-4}}\norm{\n^2 v}^2_2+C_0\xkh{\e^\be+\e^{2\be-2}}\norm{\n\nh v}^2_2\\
  &\quad+C_0\xkh{\norm{V_\e}^2_2+\norm{V_\e}^2_2\norm{\n\nh v}^2_2+\norm{V_\e}^2_2\norm{\n\nh \r}^2_2}.
\end{flalign*}Noting that the fact that $(V_\e,W_\e,\g_\e)|_{t=0}=0$, and applying the Gronwall inequality to the above inequality, it follows from Proposition \ref{po:J4t} and \ref{po:J5t} that
\begin{flalign*}
  &\xkh{\norm{\n (V_\e,\e W_\e,\g_\e)}^2_2}(t)+\ds{\xkh{\norm{\n\nh V_\e}^2_2+\e^{\be-2}\norm{\n\p_z V_\e}^2_2+\e^2\norm{\n\nh W_\e}^2_2}}\\
  &\quad\quad+\ds{\xkh{\norm{\n\nh \g_\e}^2_2+\e^\be\norm{\n\p_z W_\e}^2_2+\e^{\ga-2}\norm{\n\p_z \g_\e}^2_2}}\\
  &\quad\leq \exp\bigg\{C'\ds{\zkh{\norm{v}^2_2+\norm{\nh v}^2_2+(1+\e^2)\norm{\n v}^2_2+(1+\e^2)\norm{\n\nh v}^2_2}}\\
  &\quad\quad+C'\ds{\zkh{\e^2\norm{\n^2\nh v}^2_2+\e^4\norm{\n^2 v}^2_2\norm{\n^2\nh v}^2_2+(1+\e^4)\norm{\n v}^4_2}}\\
  &\quad\quad+C'\ds{\zkh{(1+\e^2)+\norm{v}^4_2+\norm{\n^2 v}^2_2+\norm{v}^2_2\norm{\nh v}^2_2+\norm{\n v}^2_2\norm{\nh v}^2_2}}\\
  &\quad\quad+C'\ds{\zkh{\norm{\nh \r}^2_2+\norm{\n \r}^2_2+\norm{\n \r}^4_2+(1+\e^4)\norm{\n v}^2_2\norm{\n\nh v}^2_2}}\\
  &\quad\quad+C'\ds{\xkh{\norm{\n^2 \r}^2_2+\norm{\n\nh \r}^2_2+\norm{\n \r}^2_2\norm{\n \nh \r}^2_2+\norm{V_\e}^2_2}}\\
  &\quad\quad+C'\ds{\xkh{\norm{\nh V_\e}^2_2+\e^2\norm{\nh W_\e}^2_2+\norm{V_\e}^4_2+\norm{V_\e}^2_2\norm{\nh V_\e}^2_2}}\bigg\}\\
  &\quad\quad\times\bigg\{C'\e^2\ds{\xkh{\norm{v}^2_2+\norm{\nh v}^2_2+\norm{\n v}^2_2+\norm{\n\nh v}^2_2+\norm{\n^2\nh v}^2_2}}\\
  &\quad\quad+C'\e^2\ds{\xkh{\norm{v}^4_2+\norm{v}^2_2\norm{\nh v}^2_2+\norm{\n^2 v}^2_2\norm{\n^2\nh v}^2_2+\norm{\n \p_t v}^2_2}}\\
  &\quad\quad+C'\ds{\xkh{\norm{V_\e}^4_2+\norm{V_\e}^2_2\norm{\nh V_\e}^2_2}}+C'\xkh{\e^{\ga-2}+\e^{2\ga-4}}\ds{\norm{\n^2 \r}^2_2}\\
  &\quad\quad+C'\xkh{\e^{\be-2}+\e^{2\be-4}}\ds{\norm{\n^2 v}^2_2}+C'\xkh{\e^\be+\e^{2\be-2}}\ds{\norm{\n\nh v}^2_2}\\
  &\quad\quad+C'\ds{\xkh{\norm{V_\e}^2_2+\norm{V_\e}^2_2\norm{\n\nh v}^2_2+\norm{V_\e}^2_2\norm{\n\nh \r}^2_2}}\bigg\}\\
  &\quad\leq C'\e^\eta(t+1)\zkh{\ma{J}_4(t)+\ma{J}^2_4(t)+\ma{J}_5(t)+\ma{J}^2_5(t)+\ma{J}_4(t)\ma{J}_5(t)}\\
  &\quad\quad\times\exp\dkh{C'(t+1)\zkh{\ma{J}_4(t)+\ma{J}^2_4(t)+\ma{J}_5(t)+\ma{J}^2_5(t)+1}},
\end{flalign*}where $\eta=\min\{2,\be-2,\ga-2\}$ with $2<\be,\ga<\infty$.~This completes the proof of Proposition \ref{po:J6t}.
\end{proof}

By finding a small positive constant to eliminate the effect of the smallness condition in Proposition \ref{po:J6t},~we establish the $H^1$ estimate on difference function $(V_\e,W_\e,\g_\e)$.
\begin{proposition}\label{po:J56}
Let $\ma{T}^*_\e$ be the maximal existence time of the strong solution $(v_\e,w_\e,\r_\e)$ to the system (\ref{eq:ptv}) corresponding to (\ref{ga:are})-(\ref{eq:eve}).~Then, for any finite time $\ma{T}>0$, there exists a small positive constant $\e(\ma{T})=\xkh{\frac{3\ell^2_0}{8\ma{J}_6(\ma{T})}}^{1/\eta}$ such that $\ma{T}^*_\e>\ma{T}$ provided that $\e \in $ $(0,\e(\ma{T}))$. Furthermore, the system (\ref{fl:Ve0})-(\ref{fl:zWe}) has the following energy estimate
\begin{flalign*}
  &\sup_{0 \leq t \leq \ma{T}}\xkh{\norm{(V_\e,\e W_\e,\g_\e)}^2_{H^1}}(t)+\int^\ma{T}_0{\xkh{\norm{\nh V_\e}^2_{H^1}+\e^{\be-2}\norm{\p_z V_\e}^2_{H^1}}}dt\\
  &\qquad\quad+\int^\ma{T}_0{\xkh{\e^2\norm{\nh W_\e}^2_{H^1}+\norm{\nh \g_\e}^2_{H^1}+\e^\be\norm{\p_z W_\e}^2_{H^1}+\e^{\ga-2}\norm{\p_z \g_\e}^2_{H^1}}}dt\\
  &\qquad\leq \e^\eta\xkh{\ma{J}_5(\ma{T})+\ma{J}_6(\ma{T})},
\end{flalign*}where $\eta=\min\{2,\be-2,\ga-2\}$ with $2<\be,\ga<\infty$. Here $\ma{J}_5(t)$ and $\ma{J}_6(t)$ are nonnegative continuously increasing functions that do not depend on $\e$.
\end{proposition}

\begin{proof}[Proof.]
For any finite time~$\ma{T}>0$,~setting~$\ma{T}^\delta_\e=\min\{\ma{T}^*_\e,\ma{T}\}$,~then from Proposition~\ref{po:J5t}~it follows that
\begin{flalign}
  &\sup_{0 \leq t < \ma{T}^\delta_\e}\xkh{\norm{(V_\e,\e W_\e,\g_\e)}^2_2}(t)+\int^{\ma{T}^\delta_\e}_0{\xkh{\norm{\nh V_\e}^2_2+\e^{\be-2}\norm{\p_z V_\e}^2_2}}dt\nonumber\\
  &\quad+\int^{\ma{T}^\delta_\e}_0{\xkh{\e^2\norm{\nh W_\e}^2_2+\norm{\nh \g_\e}^2_2+\e^\be\norm{\p_z W_\e}^2_2+\e^{\ga-2}\norm{\p_z \g_\e}^2_2}}dt \leq \e^\eta\ma{J}_5(\ma{T}),\label{fl:J5T}
\end{flalign}where
\begin{equation*}
  \ma{J}_5(\ma{T})=C(\ma{T}+1)e^{C(\ma{T}+1)\xkh{\ma{J}_4(\ma{T})+\ma{J}^2_4(\ma{T})}}
  \zkh{\ma{J}_4(\ma{T})+\ma{J}^2_4(\ma{T})+\xkh{\norm{v_0}^2_2+\norm{w_0}^2_2+\norm{\r_0}^2_2}^2}.
\end{equation*}Here $\eta=\min\{2,\be-2,\ga-2\}$ with $2<\be,\ga<\infty$, and $C$ is a positive constant that does not depend on $\e$.

Let~$\ell_0$~be the constant from Proposition~\ref{po:J6t}.~Define
\begin{equation*}
  t^\delta_\e:=\sup\dkh{t \in (0,\ma{T}^\delta_\e)\bigg|\sup_{0 \leq s \leq t}\xkh{\norm{\n (V_\e,\e W_\e,\g_\e)}^2_2}(s) \leq \ell^2_0}.
\end{equation*}By virtue of Proposition~\ref{po:J6t},~the following estimate holds
\begin{flalign}
  &\sup_{0 \leq s \leq t}\xkh{\norm{\n(V_\e,\e W_\e,\g_\e)}^2_2}(s)+\ds{\xkh{\norm{\n\nh V_\e}^2_2+\e^{\be-2}\norm{\n\p_z V_\e}^2_2+\e^2\norm{\n\nh W_\e}^2_2}}\nonumber\\
  &\qquad+\ds{\xkh{\norm{\n\nh \g_\e}^2_2+\e^\be\norm{\n\p_z W_\e}^2_2+\e^{\ga-2}\norm{\n\p_z \g_\e}^2_2}}\leq \e^\eta\ma{J}_6(\ma{T}),\label{fl:J6T}
\end{flalign}for any $t \in [0,t^\delta_\e)$, where
\begin{flalign*}
  \ma{J}_6(\ma{T})&\leq C(\ma{T}+1)\zkh{\ma{J}_4(\ma{T})+\ma{J}^2_4(\ma{T})+\ma{J}_5(\ma{T})+\ma{J}^2_5(\ma{T})+\ma{J}_4(\ma{T})\ma{J}_5(\ma{T})}\\
  &\quad\times\exp\dkh{C(\ma{T}+1)\zkh{\ma{J}_4(\ma{T})+\ma{J}^2_4(\ma{T})+\ma{J}_5(\ma{T})+\ma{J}^2_5(\ma{T})+1}}.
\end{flalign*}

Choosing~$\e(\ma{T})=\xkh{\frac{3\ell^2_0}{8\ma{J}_6(\ma{T})}}^{1/\eta}$,~it deduces from~(\ref{fl:J6T})~that
\begin{flalign*}
  &\sup_{0 \leq s \leq t}\xkh{\norm{\n(V_\e,\e W_\e,\g_\e)}^2_2}(s)+\ds{\xkh{\norm{\n\nh V_\e}^2_2+\e^{\be-2}\norm{\n\p_z V_\e}^2_2+\e^2\norm{\n\nh W_\e}^2_2}}\\
  &\qquad+\ds{\xkh{\norm{\n\nh \g_\e}^2_2+\e^\be\norm{\n\p_z W_\e}^2_2+\e^{\ga-2}\norm{\n\p_z \g_\e}^2_2}}\leq \frac{3\ell^2_0}{8},
\end{flalign*}for any~$t \in [0,t^\delta_\e)$~and for any~$\e \in (0,\e(\ma{T}))$,~which gives
\begin{equation}\label{eq:320}
  \sup_{0 \leq t < t^\delta_\e}\xkh{\norm{\n(V_\e,\e W_\e,\g_\e)}^2_2}(t) \leq \frac{3\ell^2_0}{8} < \ell^2_0.
\end{equation}According to the definition of~$t^\delta_\e$,~(\ref{eq:320})~implies~that $t^\delta_\e=\ma{T}^\delta_\e$.~In consequence,~the estimate~(\ref{fl:J6T})~holds for~$t \in [0,\ma{T}^\delta_\e)$~and for any~$\e \in (0,\e(\ma{T}))$.

We claim that~$\ma{T}^*_\e>\ma{T}$ for any $\e \in (0,\e(\ma{T}))$.~If $\ma{T}^*_\e \leq \ma{T}$,~then it is obvious that
\begin{equation*}
  \limsup_{t \rightarrow (\ma{T}^*_\e)^-}\xkh{\norm{\n(V_\e,\e W_\e,\g_\e)}_2}=\infty.
\end{equation*}Otherwise,~the strong solution~$(v_\e,w_\e,\r_\e)$~to the system~(\ref{eq:ptv})~can be extended beyond the maximal existence time~$\ma{T}^*_\e$.~However,~the above result contradicts to~(\ref{fl:J6T}).~This contradiction leads to $\ma{T}^*_\e>\ma{T}$, and hence $\ma{T}^\delta_\e=\ma{T}$.~Moreover,~combining~(\ref{fl:J5T})~with~(\ref{fl:J6T}) yields the energy estimate in Proposition \ref{po:J56}. This completes the proof.
\end{proof}

Based on Proposition \ref{po:J56}, we give the proof of Theorem \ref{th:J7t}.

\begin{proof}[Proof of Theorem~\ref{th:J7t}.]
For any finite time $\ma{T}>0$, by virtue of Proposition \ref{po:J56}, there exists a small positive constant $\e(\ma{T})=\xkh{\frac{3\ell^2_0}{8\ma{J}_6(\ma{T})}}^{1/\eta}$ such that $\ma{T}^*_\e>\ma{T}$ provided that $\e \in (0,\e(\ma{T}))$, which implies that the system (\ref{eq:ptv}) with (\ref{ga:are})-(\ref{eq:eve}) has a unique strong solution $(v_\e,w_\e,\r_\e)$ on the time interval $[0,\ma{T}]$ as long as $\e \in (0,\e(\ma{T}))$. Let $\ma{J}_7(\ma{T})=\ma{J}_5(\ma{T})+\ma{J}_6(\ma{T})$, then the following estimate holds
\begin{flalign*}
  &\sup_{0 \leq t \leq \ma{T}}\xkh{\norm{(V_\e,\e W_\e,\g_\e)}^2_{H^1}}(t)+\int^\ma{T}_0{\xkh{\norm{\nh V_\e}^2_{H^1}+\e^{\be-2}\norm{\p_z V_\e}^2_{H^1}}}dt\\
  &\quad+\int^\ma{T}_0{\xkh{\e^2\norm{\nh W_\e}^2_{H^1}+\norm{\nh \g_\e}^2_{H^1}+\e^\be\norm{\p_z W_\e}^2_{H^1}+\e^{\ga-2}\norm{\p_z \g_\e}^2_{H^1}}}dt
  \leq \e^\eta\ma{J}_7(\ma{T}),
\end{flalign*}where $\eta=\min\{2,\be-2,\ga-2\}$ with $2<\be,\ga<\infty$, and $\ma{J}_7(t)$ is a nonnegative continuously increasing function that does not depend on $\e$.~Finally, it is clear that the strong convergences stated in Theorem \ref{th:J7t} are the direct consequences of the above estimate. The theorem is thus proved.
\end{proof}
\smallskip\noindent
{\bf Acknowledgment.~}\small
The work of X. Pu was supported in part by the National Natural Science Foundation of China (No. 11871172) and the Natural Science Foundation of Guangdong Province of China (No. 2019A1515012000). The work of W. Zhou was supported by the Innovation Research for the Postgraduates of Guangzhou University (No. 2021GDJC-D09).

\end{document}